\documentclass[10pt, a4paper]{article}
\usepackage[utf8]{inputenc}
\usepackage[english]{babel}
\usepackage{hyperref}
\usepackage[a4paper, tmargin=1.0in,bmargin=1.0in, lmargin=0.7in, rmargin=0.7in]{geometry}
\usepackage{amssymb}
\usepackage{amsthm}
\usepackage{amsmath}
\usepackage{bm}
\usepackage{graphicx}
\usepackage{cleveref}
\usepackage{tikz}
\usepackage{authblk}
\usepackage{placeins}
\usepackage{multirow}
\usepackage{array}

\usepackage{scalerel,stackengine}

\usepackage{macros} 
\usepackage[normalem]{ulem}

\providecommand{\keywords}[1]{\textbf{\textit{Keywords---}} #1}

\title{The influence of parasitic modes on ``weakly'' unstable multi-step Finite Difference schemes}
\author{\textbf{Thomas Bellotti}\footnote{Former affiliation: CMAP, CNRS, \'Ecole polytechnique, Institut Polytechnique de Paris, 91120 Palaiseau, France.} (\href{thomas.bellotti@math.unistra.fr}{\texttt{thomas.bellotti@math.unistra.fr}})\\
IRMA, Universit\'e de Strasbourg, 67000 Strasbourg, France}

\begin{document}

\maketitle

\begin{abstract}
    Numerical analysis for linear constant-coefficients Finite Difference schemes was developed approximately fifty years ago.
    It relies on the assumption of scheme stability and in particular---for the $L^2$ setting---on the absence of multiple roots of the amplification polynomial on the unit circle.
    This allows to decouple, while discussing the convergence of the method, the study of the consistency of the scheme from the precise knowledge of its parasitic/spurious modes, so that multi-step methods can be studied essentially as they were one-step schemes. 
    In other words, the global truncation error can be inferred from the local truncation error.
    Furthermore, stability alleviates the need to delve into the complexities of floating-point arithmetic on computers, which can be challenging topics to address.
    In this paper, we show that in the case of ``weakly'' unstable schemes with multiple roots on the unit circle, although the schemes may remain stable, the consideration of parasitic modes is essential in studying their consistency and, consequently, their convergence. Otherwise said, the lack of genuine stability prevents bounding the global truncation error using the local truncation error, and one is thus compelled to study the former on its own.
    This research was prompted by unexpected numerical results on lattice Boltzmann schemes, which can be rewritten in terms of multi-step Finite Difference schemes. Initial expectations suggested that third-order initialization schemes would suffice to maintain the accuracy of a fourth-order multi-step scheme. However, this assumption proved incorrect for ``weakly'' unstable schemes. This borderline scenario underscores the significance of genuine stability in facilitating the construction of Lax-Richtmyer-like theorems and in mastering the impact of round-off errors. Despite the simplicity and apparent lack of practical usage of the linear transport equation at constant velocity considered throughout the paper, we demonstrate that high-order lattice Boltzmann schemes for this equation can be used to tackle non-linear systems of conservation laws relying on a Jin-Xin approximation and high-order splitting formul\ae{}.
\end{abstract}



\section*{Introduction}

Multi-step constant-coefficients Finite Difference schemes feature several modes \cite[Chapter 2]{gustafsson1995time}---each one associated with one root of the amplification polynomial of the scheme. For scalar problems being first-order in the time derivative, whom we shall be concerned with in this paper, one physical mode and possibly several parasitic modes are therefore mixed together in the discrete solution.
As long as the scheme is stable for the $L^2$ norm, namely the roots are in the closed unit disk and those on the unit circle are simple, stability rules out any potential influence of the---indeed present---parasitic modes, as far as consistency is concerned.
This can be understood in the following way: let $\numberSteps + 1 \leq \indicesTime \leq \lfloor \finalTime / \timeStep \rfloor$, where $\numberSteps + 1$ is the number of steps of the multi-step scheme, $\finalTime > 0$ is the final time-horizon of the simulation, and $\timeStep$ is the time step. Denote $\solutionScheme^{\indicesTime}$ the solution of the multi-step scheme, obtained by\footnote{Using the hat to denote a Fourier transform.} $\fourierTransformed{\solutionScheme}^{\indicesTime} (\frequency) = \amplificationFactorFourier{\indicesTime} (\frequency\spaceStep)\fourierTransformed{\solutionScheme}^0(\frequency)$ with $|\frequency\spaceStep| \leq \pi$, where the amplification factors $\amplificationFactorFourier{\indicesTime}$ are determined by the multi-step scheme itself as well as by the initialization schemes.
Let $\rootAmplificationPolynomialFourier{1}$ be the physical root of the amplification polynomial: the only one such that $\rootAmplificationPolynomialFourier{1}(0) = 1$, and construct $\fourierTransformed{\solutionPseudoScheme}^{\indicesTime} (\frequency) = \rootAmplificationPolynomialFourier{1}(\frequency\spaceStep)^{\indicesTime} \fourierTransformed{\solutionScheme}^0(\frequency)$, called ``pseudo-scheme''.   
The crucial estimate to study the order of the overall method (\idEst{} taking the initialization into account) is 
\begin{equation}\label{eq:stabilityDefinition}
    \weightedLTwoNorm{\solutionScheme^{\indicesTime} - \solutionPseudoScheme^{\indicesTime}} \leq \constantStability \sum_{\indicesFreeTwo = 0}^{\numberSteps} \weightedLTwoNorm{\solutionScheme^{\indicesFreeTwo} - \solutionPseudoScheme^{\indicesFreeTwo}},
\end{equation}
thanks to stability.
The right-hand side of \eqref{eq:stabilityDefinition} solely depends on $\rootAmplificationPolynomialFourier{1}$---through $\solutionPseudoScheme^{\numberSteps}, \dots, \solutionPseudoScheme^{0}$---and on the initializations $\solutionScheme^{\numberSteps}, \dots, \solutionScheme^{0}$ \emph{via} $\amplificationFactorFourier{\numberSteps}, \dots, \amplificationFactorFourier{0}$: no influence of the parasitic roots $\rootAmplificationPolynomialFourier{2}, \dots, \rootAmplificationPolynomialFourier{\numberSteps + 1}$ of the amplification polynomial.
An important consequence of this, in the context where $\timeStep \propto \spaceStep$ as $\spaceStep \to 0$ and the target equation is the linear transport equation, is that a scheme of order $\maxOrderAccuracy$ can be initialized with methods of order $\maxOrderAccuracy-1$ without lowering the overall order of the method, see \cite[Theorem 10.6.2]{strikwerda2004finite}.
For concreteness, a leap-frog scheme (\idEst{} $\maxOrderAccuracy = 2$) can be started using a Lax-Friedrichs scheme, or even an unstable forward-Euler centered scheme. 

The situation turns out to be radically different when multiple roots lay on the unit circle.
These schemes, which we call ``weakly'' unstable, can be stable in practice, especially when the frequency at which multiple roots belong to the unit circle is the frequency $\frequency = 0$ only. Therefore, provided that the underlying floating point arithmetic is accurate enough, these schemes can be suitable for computations.
Genuine stability---\idEst{} \eqref{eq:stabilityDefinition}---requires control for arbitrary initial data $\solutionScheme^{\numberSteps}, \dots, \solutionScheme^{0}$ and thus cannot be established in these circumstances, due to the potential polynomial growth of the numerical solution in $\indicesTime$. Despite this, any ``reasonable'' initialization scheme---which yields a specific choice of initial data---renders a stable computation as long as it avoids exciting the unstable frequency.
Constraining the choice of initialization routines to well-chosen ones in order to achieve desired numerical properties---which we would otherwise unlikely obtain---has been studied in the context of linear multi-step schemes for ordinary differential equations \cite{hundsdorfer2003monotonicity}.
\textbf{The main finding and the subject we try to elucidate in the present paper is the following. In the weakly unstable framework, a scheme of order $\maxOrderAccuracy$ might need to be initialized with methods of the same order $\maxOrderAccuracy$ to preserve the overall order, contrarily to genuinely stable schemes.
This is due to the entangled role of the several modes allowed by the multi-step scheme---both physical and parasitic---in the consistency of the method, since the lack of stability forbids from synthesizing consistency merely in terms of the physical root.}
To the best of our knowledge, this observation is new in the available literature.

This conclusion has been accidentally drawn while trying to construct a fourth-order lattice Boltzmann scheme \cite{mcnamara88lattice} for the linear transport equation.
Being stable, we reasonably conjectured that it would have needed third-order initializations in order to preserve its overall fourth-order.
Surprisingly, despite the fact of having third-order initializations ready to use, numerical results featured orders of convergence stuck at three instead of the expected order four, without any trace of explosion of the numerical solution whatsoever.
This fact indicates that the initialization of stable lattice Boltzmann schemes must be handled with particular care.
The majority of the content in the present paper is likely to be essentially useless in applications and belatedly tweaks subjects that have been studied a long time ago, at the beginning of the 1970s \cite{brenner1970stability}. Still, it emphasizes the key role of stability in making the study of consistency simple, \idEst{} based just on the local truncation error of the bulk multi-step scheme, plus the independent study of the initialization routines (\confer{} \eqref{eq:stabilityDefinition}). 
Indeed, these parts of the contribution should be valued for their pedagogical role.
However, developing high-order lattice Boltzmann schemes for the linear transport equation can be of interest in applications.
Indeed, these explicit methods are praised for their computational efficiency and can be used to approximate the solution of non-linear systems of conservation laws by embedding them into approximations \cite{aregba2000discrete} of the Jin-Xin relaxation system \cite{jin1995relaxation}, using high-order symmetric operator splittings \cite{mclachlan2002splitting}, in the spirit of \cite{coulette2019high}. 

The paper is structured as follows.
\Cref{sec:originProblem} presents the origin of the study from an empirical observation, where we found unexpected orders of convergence.
Theory is developed in \Cref{sec:understandingTheory} to understand these results.
\Cref{sec:roundoffErrors} brings the numerical schemes back on the ``battlefield'' by investigating the role of floating-point arithmetic and thus of round-off errors.
Finally, \Cref{sec:applicationKinetic} presents an application of the fourth-order solver for the linear transport equation to the approximation of non-linear equations.

The model system we consider in the paper is the Cauchy problem associated with the linear transport equation at velocity $\advectionVelocity \in \reals$, which reads
\begin{align}
    \partial_{\timeVariable} \solution(\timeVariable, \spaceVariable) + \advectionVelocity \partial_{\spaceVariable} \solution(\timeVariable, \spaceVariable) &= 0, \qquad &&\timeVariable \in (0, \finalTime], \quad &&\spaceVariable\in \reals, \label{eq:transportEquation} \\
    \solution(0, \spaceVariable) &= \initialDatum (\spaceVariable), &&\qquad &&\spaceVariable \in \reals. \label{eq:initialDatum}
\end{align}
The initial datum $\initialDatum$ is assumed to be a given smooth function, unless otherwise said.
For the problem is linear, we conveniently consider the Fourier transform of \eqref{eq:transportEquation}, obtaining $\partial_{\timeVariable} \fourierTransformed{\solution}(\timeVariable, \frequency)  = -\imaginaryUnit\advectionVelocity\frequency \fourierTransformed{\solution}(\timeVariable, \frequency)$ for $\frequency \in \reals$.
The explicit solution hence reads $\fourierTransformed{\solution}(\timeVariable, \frequency) = e^{-\imaginaryUnit\advectionVelocity \frequency \timeVariable} {\initialDatumFourier}(\frequency)$.

For the sake of approximating the solution of \eqref{eq:transportEquation} and \eqref{eq:initialDatum}, we consider a uniform time-space discretization with steps $\timeStep$ and $\spaceStep$, so that the discrete grid points in time will be $\gridPointTime{\indicesTime} \definitionEquality \indicesTime \timeStep$ with $\indicesTime \in \naturals$ and those in space $\gridPointSpace{\indicesSpace} \definitionEquality \indicesSpace \spaceStep$ with $\indicesSpace \in \relatives$.
For we consider explicit numerical methods for the hyperbolic equation \eqref{eq:transportEquation}, we naturally fix the ratio\footnote{We acknowledge that this notation is ``standard'' in the lattice Boltzmann community, while the Finite Difference community mostly employs the reciprocal $\latticeVelocity = \timeStep/\spaceStep$. We shall stick with the former notation.} $\latticeVelocity \definitionEquality \spaceStep/\timeStep$ to some positive real number as the spatial grid shrinks, \idEst{} $\spaceStep \to 0$, and so we shall be allowed to use $\spaceStep$ as unique discretization parameter.

\section{The numerical experiment behind this study}\label{sec:originProblem}

As extensively stressed in the introduction, this study has been stimulated by the following example of multi-step scheme with $\numberSteps + 1 = 3$ steps, which origin will be clarified in a moment.
Let $\indicesTime \geq 2$, $\indicesSpace \in \relatives$, and consider:
\begin{align}
    \solutionDiscrete{\indicesTime + 1}{\indicesSpace} =  &\tfrac{1}{3} (1 - 4\courantNumber^2 + 2(\courantNumber^2 - 1) (\centralSecondDifference + 2) - 6\courantNumber\centralFirstDifference) \solutionDiscrete{\indicesTime}{\indicesSpace} \nonumber \\
    - &\tfrac{1}{3} (1 - 4\courantNumber^2 + 2(\courantNumber^2 - 1) (\centralSecondDifference + 2) + 6\courantNumber\centralFirstDifference)  \solutionDiscrete{\indicesTime - 1}{\indicesSpace} + \solutionDiscrete{\indicesTime - 2}{\indicesSpace}, \label{eq:FourthOrderLeapFrog}
\end{align}
where the centered first-order finite difference $\centralFirstDifference$ is defined by $\centralFirstDifference \solutionDiscrete{ }{\indicesSpace} \definitionEquality (\solutionDiscrete{ }{\indicesSpace + 1} - \solutionDiscrete{ }{\indicesSpace - 1})/2$, and the centered second-order finite difference $\centralSecondDifference$ is given by $\centralSecondDifference \solutionDiscrete{ }{\indicesSpace} \definitionEquality \solutionDiscrete{ }{\indicesSpace + 1}-2\solutionDiscrete{ }{\indicesSpace} + \solutionDiscrete{ }{\indicesSpace-1}$.
For it will play an important role in what follows, we define the Courant number $\courantNumber\definitionEquality \advectionVelocity/\latticeVelocity$.
In \eqref{eq:FourthOrderLeapFrog}, we have to interpret $\solutionDiscrete{\indicesTime}{\indicesSpace} \approx \solution(\gridPointTime{\indicesTime}, \gridPointSpace{\indicesSpace})$.

\subsection{Construction of \eqref{eq:FourthOrderLeapFrog} from a lattice Boltzmann scheme}

The multi-step Finite Difference scheme \eqref{eq:FourthOrderLeapFrog} is constructed starting from a lattice Boltzmann scheme, as previously described in \cite{bellotti2022finite}.
Let us emphasize that---though the theory on Finite Difference schemes has been understood for decades now---recent works in the framework of lattice Boltzmann schemes \cite{fuvcik2021equivalent, bellotti2022finite} have demonstrated that these latter are indeed an inextinguishable source of multi-step Finite Difference schemes.
Therefore the (still imperfect) rigorous understanding of lattice Boltzmann schemes is likely to pass from that of multi-step Finite Difference schemes \cite{bellotti2023truncation}.

\subsubsection{Lattice Boltzmann algorithm: collide-and-stream }\label{sec:collideAndStreamAlgorithm}

Consider a scheme in one space dimension featuring three discrete velocities, see for example \cite{dubois2013stable}.
Without entering into the details on how lattice Boltzmann schemes are devised, we just consider that they stem from a collide-and-stream procedure made up as follows.
Here, $\indicesTime \in \naturals$.
\begin{itemize}
    \item The \emph{local collision phase} reads, for $\indicesSpace \in \relatives$:
    \begin{equation}\label{eq:collisionLatticeBoltzmann}
        \firstMomentDiscrete{\indicesTime\collided}{\indicesSpace} = \firstMomentDiscrete{\indicesTime}{\indicesSpace}, \qquad \secondMomentDiscrete{\indicesTime\collided}{\indicesSpace} = (1-\relaxationParameterSecondMoment) \secondMomentDiscrete{\indicesTime}{\indicesSpace} + \relaxationParameterSecondMoment \secondMomentLetter^{\atEquilibrium} ( \firstMomentDiscrete{\indicesTime}{\indicesSpace}), \qquad \thirdMomentDiscrete{\indicesTime\collided}{\indicesSpace} = (1-\relaxationParameterThirdMoment) \thirdMomentDiscrete{\indicesTime}{\indicesSpace} + \relaxationParameterThirdMoment \thirdMomentLetter^{\atEquilibrium} ( \firstMomentDiscrete{\indicesTime}{\indicesSpace}).
    \end{equation}
    In these expressions, $\relaxationParameterSecondMoment, \relaxationParameterThirdMoment \in (0, 2]$ are the relaxation parameters of the non-conserved moments $\secondMomentDiscrete{}{}$ and $\thirdMomentDiscrete{}{}$, whereas $\secondMomentLetter^{\atEquilibrium}$ and $\thirdMomentLetter^{\atEquilibrium}$ are their equilibria: possibly non-linear functions of the conserved moment $\firstMomentDiscrete{}{}$.
    For the problem we aim at solving is linear, we consider linear equilibria and thus take $\equilibriumCoefficientSecond, \equilibriumCoefficientThird \in \reals$ such that $\secondMomentLetter^{\atEquilibrium}(\firstMomentDiscrete{}{}) = \equilibriumCoefficientSecond \firstMomentDiscrete{}{}$ and $\thirdMomentLetter^{\atEquilibrium}(\firstMomentDiscrete{}{}) = \equilibriumCoefficientThird \firstMomentDiscrete{}{}$.
    \item The \emph{non-local stream phase} is written using another basis induced by $\momentMatrix^{-1}$.
    We take 
    \begin{equation*}
        \momentMatrix =
        \begin{bmatrix}
            \momentMatrixEntryLetter_{11} & 1 & 1 \\
            0 & 1 & -1 \\
            \momentMatrixEntryLetter_{31} & 1 & 1
        \end{bmatrix},
    \end{equation*} 
    where $\momentMatrixEntryLetter_{11}, \momentMatrixEntryLetter_{31} \in \reals$ remain free parameters that could be tuned to change some features of the scheme, in particular stability.
    To keep the matrix $\momentMatrix$ invertible, we have to enforce $\momentMatrixEntryLetter_{11} \neq \momentMatrixEntryLetter_{31}$.
    The usual choice \cite{dubois2013stable, fevrier2014extension, dubois2020notion} is to take $\momentMatrixEntryLetter_{11} = 1$, that is, if we interpret the first moment $\firstMomentDiscrete{}{}$ as a density, all the ``particles'' have the same mass. It is natural to take $\momentMatrixEntryLetter_{12} = \momentMatrixEntryLetter_{13}$ for symmetry reasons, and we take these two entries equal to one as a normalization.
    It is also natural to consider $\momentMatrixEntryLetter_{21} = 0$, to avoid linear transport terms which do not originate from the equilibria, being intrinsically linear.
    Usual values for $\momentMatrixEntryLetter_{31}$ are zero \cite{fevrier2014extension} (for simplicity) and -2 \cite{dubois2020notion} (to have orthogonal rows in $\momentMatrix$ with respect to the Euclidean scalar product of vectors). Again, $ \momentMatrixEntryLetter_{32} =  \momentMatrixEntryLetter_{33} = 1$ by symmetry and normalization arguments.
    The post-collisional distribution functions associated with the discrete velocities $0$, $1$, and $-1$ are recovered point-by-point by $\transpose{(\distributionFunctionDiscreteZero{\indicesTime\collided}{\indicesSpace}, \distributionFunctionDiscretePlus{\indicesTime\collided}{\indicesSpace}, \distributionFunctionDiscreteMinus{\indicesTime\collided}{\indicesSpace})} = \momentMatrix^{-1} \transpose{(\firstMomentDiscrete{\indicesTime\collided}{\indicesSpace}, \secondMomentDiscrete{\indicesTime\collided}{\indicesSpace}, \thirdMomentDiscrete{\indicesTime\collided}{\indicesSpace})}$, where $\indicesSpace \in \relatives$.
    The stream reads, for $\indicesSpace \in \relatives$
    \begin{equation}\label{eq:streamLatticeBoltzmann}
        \distributionFunctionDiscreteZero{\indicesTime + 1}{\indicesSpace} = \distributionFunctionDiscreteZero{\indicesTime\collided}{\indicesSpace}, \qquad \distributionFunctionDiscretePlus{\indicesTime + 1}{\indicesSpace} = \distributionFunctionDiscretePlus{\indicesTime\collided}{\indicesSpace - 1}, \qquad \distributionFunctionDiscreteMinus{\indicesTime + 1}{\indicesSpace} = \distributionFunctionDiscreteMinus{\indicesTime\collided}{\indicesSpace + 1}.
    \end{equation}
    After this phase, one can recover the moments by setting $\transpose{(\firstMomentDiscrete{\indicesTime+1}{\indicesSpace}, \secondMomentDiscrete{\indicesTime+1}{\indicesSpace}, \thirdMomentDiscrete{\indicesTime+1}{\indicesSpace})} = \momentMatrix \transpose{(\distributionFunctionDiscreteZero{\indicesTime + 1}{\indicesSpace}, \distributionFunctionDiscretePlus{\indicesTime + 1}{\indicesSpace}, \distributionFunctionDiscreteMinus{\indicesTime + 1}{\indicesSpace})}$.
\end{itemize}

The lattice Boltzmann scheme can be written on the moments using $\schemeMatrix$, a 3-by-3 matrix with entries on the ring of spatial Finite Difference operators on Cartesian grid, so that $\transpose{(\firstMomentDiscrete{\indicesTime+1}{\indicesSpace}, \secondMomentDiscrete{\indicesTime+1}{\indicesSpace}, \thirdMomentDiscrete{\indicesTime+1}{\indicesSpace})} = \schemeMatrix \transpose{(\firstMomentDiscrete{\indicesTime}{\indicesSpace}, \secondMomentDiscrete{\indicesTime}{\indicesSpace}, \thirdMomentDiscrete{\indicesTime}{\indicesSpace})} $.
Therefore $\transpose{(\firstMomentDiscrete{\indicesTime}{\indicesSpace}, \secondMomentDiscrete{\indicesTime}{\indicesSpace}, \thirdMomentDiscrete{\indicesTime}{\indicesSpace})} = \schemeMatrix^{\indicesTime} \transpose{(\firstMomentDiscrete{0}{\indicesSpace}, \secondMomentDiscrete{0}{\indicesSpace}, \thirdMomentDiscrete{0}{\indicesSpace})} $ and by the Parseval's identity, we naturally introduce the following definition of stability.
\begin{definition}[Stability of a lattice Boltzmann scheme]
    A lattice Boltzmann scheme, such as \eqref{eq:collisionLatticeBoltzmann}/\eqref{eq:streamLatticeBoltzmann}, is said to be ``stable'' if and only if $\fourierTransformed{\schemeMatrix}(\frequency\spaceStep)^{\indicesTime}$ is bounded for every $|\frequency\spaceStep|\leq \pi$ and for every $\indicesTime \in \naturals$.
\end{definition}
\begin{proposition}[Stability of a lattice Boltzmann scheme]
    A lattice Boltzmann scheme, such as \eqref{eq:collisionLatticeBoltzmann}/\eqref{eq:streamLatticeBoltzmann}, is stable if and only if,  for every $|\frequency\spaceStep|\leq \pi$, the minimal polynomial of $\fourierTransformed{\schemeMatrix}(\frequency\spaceStep)$ is a simple \emph{von Neumann} polynomial, namely none of its roots is outside the closed unit disk and those on the unit circle are simple.
\end{proposition}
\begin{proof}
    This is a consequence of the Jordan canonical form for complex matrices. If the minimal polynomial of $\fourierTransformed{\schemeMatrix}(\frequency\spaceStep)$ is a simple \emph{von Neumann} polynomial, then the maximal size of the Jordan blocks associated to each eigenvalue on the unit circle is one, which prevents polynomial growths of $\fourierTransformed{\schemeMatrix}(\frequency\spaceStep)^{\indicesTime}$ in $\indicesTime$. Exponential growths are not possible since all the roots are in the closed unit disk.
\end{proof}

\subsubsection{Tuning of the free parameters}\label{sec:tuningFreeParam}

Looking at the algorithm proposed in \Cref{sec:collideAndStreamAlgorithm}, we see that it features a large number of free parameters to be tuned, namely $\relaxationParameterSecondMoment$, $\relaxationParameterThirdMoment$, $\equilibriumCoefficientSecond$, $\equilibriumCoefficientThird$, $\momentMatrixEntryLetter_{11}$, and $\momentMatrixEntryLetter_{31}$.
We now select them according to the order of accuracy that we want to achieve with respect to the target equation \eqref{eq:transportEquation}.
This could also be obtained by turning the lattice Boltzmann scheme---originally on $\firstMomentDiscrete{}{}$, $\secondMomentDiscrete{}{}$, and $\thirdMomentDiscrete{}{}$---into a multi-step Finite Difference scheme solely on $\firstMomentDiscrete{}{}$, see \cite{bellotti2022finite}, and then computing the modified equations \cite{warming1974modified} or expanding the roots of the amplification polynomial in the small wave-number limit. 
Instead, we compute the modified equations on the original lattice Boltzmann scheme following the procedure proposed in \cite{dubois2022nonlinear}, where they are called ``equivalent equations''.
We assume that all the parameters remain fixed as $\spaceStep$---and \emph{a fortiori} $\timeStep$---go to zero. 
The obtained modified equation reads $\partial_{\timeVariable} \testFunction + \duboisOrderTerm{1} (\testFunction) + \sum_{\indicesOrder = 2}^{\indicesOrder = +\infty} \timeStep^{\indicesOrder - 1} \duboisOrderTerm{\indicesOrder} (\testFunction) = 0$, see \cite[Equation (38)]{dubois2009towards}, where a generic function $\testFunction = \testFunction (\timeVariable, \spaceVariable)$ appears to stress the fact that this is not the solution $\solution$ of the target problem \eqref{eq:transportEquation} and \eqref{eq:initialDatum}.
The determination of $\duboisOrderTerm{1} (\testFunction)$, enforcing that $\duboisOrderTerm{1} (\testFunction) = \advectionVelocity \partial_{\spaceVariable} \testFunction$ secures first-order consistency with \eqref{eq:transportEquation}. 
Obtaining $\duboisOrderTerm{2} (\testFunction) = \dots = \duboisOrderTerm{\maxOrderAccuracy} (\testFunction) = 0$ ensures accuracy up to order $\maxOrderAccuracy$.
We proceed iteratively order-by-order and progressively incorporate any previous choice on the parameters.
\begin{itemize}
    \item We obtain $\duboisOrderTerm{1} (\testFunction) = \latticeVelocity \equilibriumCoefficientSecond \partial_{\spaceVariable}\testFunction$. To achieve consistency, we have to enforce $\equilibriumCoefficientSecond = \courantNumber$.
    \item We have 
    \begin{equation*}
        \duboisOrderTerm{2} (\testFunction) = \latticeVelocity^2 \Bigl ( \frac{1}{\relaxationParameterSecondMoment} - \frac{1}{2} \Bigr ) \Bigl (- \frac{\momentMatrixEntryLetter_{31}}{\momentMatrixEntryLetter_{31} - \momentMatrixEntryLetter_{11}} + \courantNumber^2 + \frac{\momentMatrixEntryLetter_{11}}{\momentMatrixEntryLetter_{31} - \momentMatrixEntryLetter_{11}}\equilibriumCoefficientThird  \Bigr ) \partial_{\spaceVariable\spaceVariable} \testFunction.
    \end{equation*}
    There are two ways of having $\duboisOrderTerm{2} (\testFunction) = 0$ by making each term into parentheses vanish. We adopt $\relaxationParameterSecondMoment = 2$.
    \item We obtain 
    \begin{equation*}
        \duboisOrderTerm{3} (\testFunction) = \frac{\latticeVelocity^3 \courantNumber}{12} \Bigl ( -2\courantNumber^2 + \frac{(1-3\equilibriumCoefficientThird )\momentMatrixEntryLetter_{11} + \momentMatrixEntryLetter_{31} }{\momentMatrixEntryLetter_{31} - \momentMatrixEntryLetter_{11}}\Bigr ) \partial_{\spaceVariable}^3 \testFunction.
    \end{equation*}
    We achieve $\duboisOrderTerm{3} (\testFunction) = 0$ through $\equilibriumCoefficientThird = \tfrac{1}{3} \Bigl ( 1 + \frac{2\momentMatrixEntryLetter_{31}}{\momentMatrixEntryLetter_{11}}- 2 \frac{\momentMatrixEntryLetter_{31}-\momentMatrixEntryLetter_{11}}{\momentMatrixEntryLetter_{11}}\courantNumber^2 \Bigr )$.
    \item We have 
    \begin{equation*}
        \duboisOrderTerm{4} =\frac{\latticeVelocity^4 \courantNumber^2 (\courantNumber^2 - 1)}{6} \Bigl ( \frac{1}{\relaxationParameterThirdMoment} - \frac{1}{2}\Bigr )\partial_{\spaceVariable}^4 \testFunction,
    \end{equation*}
    hence achieve fourth-order accuracy by selecting $\relaxationParameterThirdMoment = 2$.
\end{itemize}

After this procedure, the coefficients $\momentMatrixEntryLetter_{11}$ and $\momentMatrixEntryLetter_{31}$ are still free.
Nevertheless, we are about to see that they do not play any major role in the rest of the paper and we can therefore fix them at our convenience.
We finish on the stability of the lattice Boltzmann scheme.
\begin{proposition}[Stability of the lattice Boltzmann scheme \eqref{eq:collisionLatticeBoltzmann}/\eqref{eq:streamLatticeBoltzmann}]
    The lattice Boltzmann scheme \eqref{eq:collisionLatticeBoltzmann}/\eqref{eq:streamLatticeBoltzmann} with the previously selected parameters is stable under the CFL condition $|\courantNumber| < 1/2$.
\end{proposition}
\begin{proof}
    For $\frequency\spaceStep \neq 0$, \Cref{prop:StabilityThirdLeapFrog} to come ensures that the characteristic polynomial of $\fourierTransformed{\schemeMatrix}(\frequency\spaceStep)$ is a simple \emph{Von Neumann} polynomial, thus this is also true for the minimal polynomial.
    For $\frequency\spaceStep = 0$, we have
    \begin{equation*}
        \fourierTransformed{\schemeMatrix}(0) = 
        \begin{bmatrix}
            1 & 0 & 0\\
            \star & - 1 & 0 \\
            \star & 0 & -1
        \end{bmatrix}, \qquad \textnormal{hence} \qquad 
        \fourierTransformed{\schemeMatrix}(0)^{\indicesTime} = 
        \begin{bmatrix}
            1 & 0 & 0\\
            \star (1-(-1)^{\indicesTime}) & (-1)^{\indicesTime} & 0 \\
            \star (1-(-1)^{\indicesTime}) & 0 & (-1)^{\indicesTime}
        \end{bmatrix},
    \end{equation*}
    where the $\star$ entries are terms depending on $\courantNumber$ and the choice of $\momentMatrixEntryLetter_{11}$ and $\momentMatrixEntryLetter_{31}$. They are independent of $\indicesTime$.
    Thus $\fourierTransformed{\schemeMatrix}(0)$ is power bounded and thus the lattice Boltzmann scheme stable.
    Observe that $\fourierTransformed{\schemeMatrix}(0)^{2} = \identityMatrix$, thus the polynomial $\timeShiftOperator^2 - 1$ is the minimal polynomial of $\fourierTransformed{\schemeMatrix}(0)$. 
    It has two roots on the unit circle which are distinct.
\end{proof}

\subsubsection{Corresponding Finite Difference scheme}

Having a fourth-order lattice Boltzmann scheme at our disposal, we now describe how to forget about it and obtain \eqref{eq:FourthOrderLeapFrog}.
In our previous contributions \cite{bellotti2022finite, bellotti2023truncation, bellotti2023initialisation}, we have shown how to recast any lattice Boltzmann scheme---whether it tackles linear or non-linear equation---under the form of a multi-step Finite Difference scheme on the conserved moments.
In the present context, the corresponding Finite Difference scheme will be on $\firstMomentDiscrete{}{}$ only.
The amplification polynomial of the corresponding Finite Difference scheme reads
\begin{equation}\label{eq:amplificationPolynomialFourthOrderLeapFrog}
    \amplificationPolynomial{\frequency\spaceStep}{\timeShiftOperator} = \determinant{\timeShiftOperator \identityMatrix - \fourierTransformed{\schemeMatrix}(\frequency\spaceStep)} = \timeShiftOperator^3 + \commonOperator(\frequency\spaceStep) \timeShiftOperator^2 - \conjugate{\commonOperator} (\frequency\spaceStep)\timeShiftOperator - 1, 
\end{equation}
where the over-line denotes complex conjugation, $\commonOperator(\frequency\spaceStep) = -\tfrac{1}{3} (1 - 4 \courantNumber^2 + 4(\courantNumber^2 - 1) \cos(\frequency\spaceStep) - 6\imaginaryUnit \courantNumber \sin(\frequency\spaceStep))$, and $|\frequency\spaceStep| \leq \pi$.
This is readily the amplification polynomial associated with \eqref{eq:FourthOrderLeapFrog}.
Observe that this scheme does not depend on the specific choice of $\momentMatrixEntryLetter_{11}$ and $\momentMatrixEntryLetter_{31}$, as previously claimed.

\subsubsection{Initialization}

We fix $\momentMatrixEntryLetter_{11} = 1$ and $\momentMatrixEntryLetter_{31} = -2$ for simplicity.
Since the bulk lattice Boltzmann scheme is fourth-order accurate, it needs to be initialized with at least third-order accurate schemes dictated by the choice of $\secondMomentDiscrete{0}{}$ and $\thirdMomentDiscrete{0}{}$.
The classical choice of taking them locally at equilibrium, namely selecting $\secondMomentDiscrete{0}{\indicesSpace} = \equilibriumCoefficientSecond \firstMomentDiscrete{0}{\indicesSpace} = \courantNumber \firstMomentDiscrete{0}{\indicesSpace}$ and $\thirdMomentDiscrete{0}{\indicesSpace}= \equilibriumCoefficientThird \firstMomentDiscrete{0}{\indicesSpace} = (2\courantNumber^2 - 1)\firstMomentDiscrete{0}{\indicesSpace}$ is not enough, because it yields first-order accurate solutions, see \cite{bellotti2023initialisation}.
Considering 
\begin{equation}\label{eq:FourthOrderInitializationLBM}
    \secondMomentDiscrete{0}{\indicesSpace} = \courantNumber \firstMomentDiscrete{0}{\indicesSpace} + \frac{\courantNumber^2 - 1}{6} \centralFirstDifference \firstMomentDiscrete{0}{\indicesSpace} , \qquad \thirdMomentDiscrete{0}{\indicesSpace} = (2\courantNumber^2 - 1)\firstMomentDiscrete{0}{\indicesSpace} + \courantNumber(\courantNumber^2 - 1)\centralFirstDifference \firstMomentDiscrete{0}{\indicesSpace} + \delta \centralSecondDifference \firstMomentDiscrete{0}{\indicesSpace}
\end{equation}
gives a third-order initialization scheme for the first stage and a third-order, for $\delta \neq 0$, or fourth-order, for $\delta = 0$, initialization scheme for the second stage by slightly perturbing the local equilibrium.
The amplification factors of the corresponding Finite Difference schemes will be given by $\amplificationFactorFourier{1}(\frequency\spaceStep) = \transpose{\canonicalBasisVector{1}} \fourierTransformed{\schemeMatrix}(\frequency\spaceStep) \transpose{(1, \courantNumber + \tfrac{\imaginaryUnit(\courantNumber^2-1)}{6}\sin(\frequency\spaceStep), (2\courantNumber^2 - 1)+{\imaginaryUnit \courantNumber(\courantNumber^2-1)}\sin(\frequency\spaceStep) + 2\delta (\cos(\frequency\spaceStep)-1))}$ and $\amplificationFactorFourier{2}(\frequency\spaceStep) = \transpose{\canonicalBasisVector{1}} \fourierTransformed{\schemeMatrix}(\frequency\spaceStep)^2 \transpose{(1, \courantNumber + \tfrac{\imaginaryUnit(\courantNumber^2-1)}{6}\sin(\frequency\spaceStep), (2\courantNumber^2 - 1)+{\imaginaryUnit \courantNumber(\courantNumber^2-1)}\sin(\frequency\spaceStep) + 2\delta (\cos(\frequency\spaceStep)-1))}$ and feature quite involved expressions that we do not provide here.
Still, we have the expansions $\amplificationFactorFourier{1}(\frequency\spaceStep) = e^{-\imaginaryUnit\courantNumber\frequency\spaceStep (1+\bigO{|\frequency\spaceStep|^3})}$ and $\amplificationFactorFourier{2}(\frequency\spaceStep) = e^{-2\imaginaryUnit\courantNumber\frequency\spaceStep (1+(\delta - 1)\bigO{|\frequency\spaceStep|^3} + \bigO{|\frequency\spaceStep|^4})}$ in the limit $|\frequency\spaceStep| \ll 1$.

\subsubsection{A surprising numerical experiment}\label{sec:firstSurprisingNumericalExperiment}

\begin{figure}[h]
    \begin{center}
        \includegraphics[width = 0.49\textwidth]{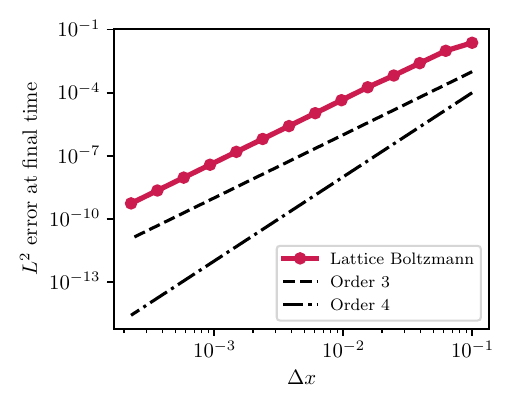}
        \includegraphics[width = 0.49\textwidth]{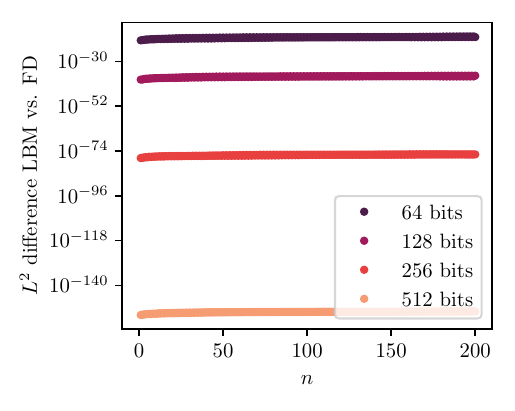}
    \end{center}\caption{\label{fig:LBM_vs_FD_arbitrary_precision}Left: error for the lattice Boltzmann scheme with initialization \eqref{eq:FourthOrderInitializationLBM} with $\delta = 1$ at final time $\finalTime = 0.2$. Right: difference between lattice Boltzmann and Finite Difference scheme  with $\numberSpacePoint = 200$ mesh points as time goes on for different floating point arithmetics.}
\end{figure}

We now test the order of convergence of the lattice Boltzmann scheme with respect to $\spaceStep$.
We simulate on a bounded domain $[-1, 1]$, enforcing periodic boundary conditions.
We employ the original lattice Boltzmann scheme \eqref{eq:collisionLatticeBoltzmann}/\eqref{eq:streamLatticeBoltzmann} with \eqref{eq:FourthOrderInitializationLBM} using $\delta = 1$, hence having third-order initializations. 
The initial datum is $\initialDatum(\spaceVariable) = \textnormal{exp}(-1/(1-(2 \spaceVariable)^2)) \indicatorFunction{(-1, 1)}(2\spaceVariable)$, which is a smooth function of class $C_{\textnormal{c}}^{\infty}([-1, 1])$ fulfilling the periodic boundary conditions.
We simulate using $\courantNumber = 1/4$.
Surprisingly, the result on the left of \Cref{fig:LBM_vs_FD_arbitrary_precision} shows third-order convergence instead of the expected fourth-order.
We will come back to this fact in a few moments.

\subsubsection{Equivalence of the lattice Boltzmann scheme and its corresponding Finite Difference scheme}\label{sec:equivalence}

Upon taking the initialization procedures into account, the unknowns $\firstMomentDiscrete{\indicesTime}{}$ computed using the original lattice Boltzmann method \eqref{eq:collisionLatticeBoltzmann}/\eqref{eq:streamLatticeBoltzmann} or its corresponding Finite Difference scheme \eqref{eq:FourthOrderLeapFrog} are mathematically the same.
Of course, since the operations implemented on computers can be different, this is actually true up to machine precision. 
To demonstrate this fact, we follow the illustration by \cite{dellar2023magic} and adopt the same setting of \Cref{sec:firstSurprisingNumericalExperiment} with a grid made up of $\numberSpacePoint = 200$ points, using both the original lattice Boltzmann scheme and its corresponding Finite Difference scheme with different machine precisions.
The results on the right of  \Cref{fig:LBM_vs_FD_arbitrary_precision} confirm our claim: the difference is of the order of the machine epsilon and accumulates in time.


\subsection{An (even more) surprising numerical experiment}\label{sec:NumericalExperiment}

\begin{figure}[h]
    \begin{center}
        \includegraphics{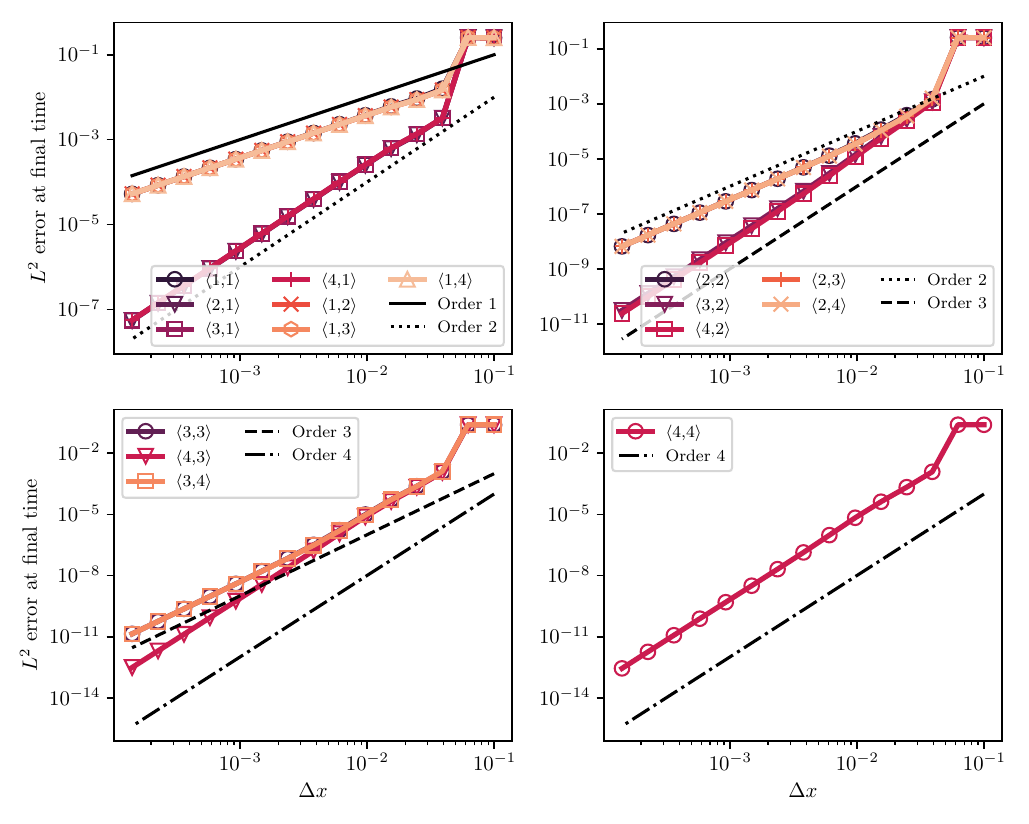}
    \end{center}\caption{\label{fig:leapFrogLBMConvergence}Error for \eqref{eq:FourthOrderLeapFrog} at final time $\finalTime = 0.2$ with different initialization schemes.}
\end{figure}

We now test the order of convergence of the corresponding Finite Difference scheme \eqref{eq:FourthOrderLeapFrog} with respect to $\spaceStep$.
The setting is the same as \Cref{sec:firstSurprisingNumericalExperiment}.
We vary a large number of initialization schemes. 
In particular, we shall use the Lax-Friedrichs scheme as prototype of first-order initialization scheme, the Lax-Wendroff scheme for second-order, the OS3 scheme for third-order, and the OS4 scheme for fourth order, see \cite{daru2004high}.
Test cases will be numbered as follows: imagine to deal with a bulk scheme featuring $\numberSteps + 1= 3$ steps, hence needing $\numberSteps= 2$ initialization schemes. The test case $\testCaseId{\maximumOrder_2, \maximumOrder_1}$ corresponds to a $\maximumOrder_2$-order scheme (applied twice) for the second initialization and a $\maximumOrder_1$-order scheme for the first initialization.
In the results given in \Cref{fig:leapFrogLBMConvergence} for the $L^2$ error at $\finalTime$, we observe an unexpected result:
\begin{equation*}
    \textnormal{overall order} = \min (\maximumOrder, \maximumOrder_2, \maximumOrder_1 + 1), \quad \textnormal{\emph{in lieu} of the expected} \quad
    \textnormal{overall order} = \min (\maximumOrder, \maximumOrder_2 + 1, \maximumOrder_1 + 1)
\end{equation*}
from \cite[Theorem 10.6.2]{strikwerda2004finite}, where the order of the bulk scheme is $\maximumOrder$ (in our case $\maximumOrder = 4$).
This means that we have to feed the second time step $\indicesTime = 2$ with an initialization of the same order of accuracy as the bulk scheme to preserve the overall order.
Observe that the computations shown in \Cref{fig:leapFrogLBMConvergence} remain stable (and converge).

\section{Understanding the numerical experiment}\label{sec:understandingTheory}

We now try to investigate the reasons behind all these unexpected results.
To this end, we first show that the Finite Difference scheme is weakly unstable and features---along with the original lattice Boltzmann scheme---two travelling parasitic modes whose speed of propagation sets a specific CFL constraint different from the usual $|\courantNumber| \leq 1$.
Then, we theoretically study the order of convergence by proving that thanks to the initialization, the scheme remains stable, and that the global truncation error is shaped by the parasitic modes.
They also allow to study the qualitative behavior of the error in time and explain supra-convergent results in simulations run under periodic boundary conditions.
Finally, we numerically analyze the case where the initial datum $\initialDatum$ is not smooth.

\subsection{Weak instability for general data}

The stability of \eqref{eq:FourthOrderLeapFrog} needs to be studied using the roots of its amplification polynomial \eqref{eq:amplificationPolynomialFourthOrderLeapFrog}.
\begin{definition}[Stability/weak instability of a Finite Difference scheme]
    Consider a Finite Difference scheme explicitly independent of $\timeStep$ and $\spaceStep$ with amplification polynomial $\amplificationPolynomial{\frequency\spaceStep}{\timeShiftOperator}$.
    \begin{itemize}
        \item We say that the scheme is ``stable'' if, for every $|\frequency\spaceStep| \leq \pi$, $\amplificationPolynomial{\frequency\spaceStep}{\timeShiftOperator}$ is a simple \emph{von Neumann} polynomial.
        \item We say that the scheme is ``weakly unstable'' if for every $|\frequency\spaceStep| \leq \pi$, $\amplificationPolynomial{\frequency\spaceStep}{\timeShiftOperator}$ is a \emph{von Neumann} polynomial, namely all the roots are inside the closed unit disk, and for some $|\tilde{\frequency}\spaceStep| \leq \pi$, $\amplificationPolynomial{\tilde{\frequency}\spaceStep}{\timeShiftOperator}$ has multiple roots on the unit circle.
    \end{itemize}
\end{definition}

If we compute its conjugate reciprocal polynomial (or inversive polynomial) \cite{vieira2019polynomials} given by $\amplificationPolynomialConjugate{\frequency\spaceStep}{\timeShiftOperator} \definitionEquality \timeShiftOperator^{\numberSteps + 1} \amplificationPolynomial{-\frequency\spaceStep}{1/\timeShiftOperator}$, we obtain
\begin{equation*}
    \amplificationPolynomialConjugate{\frequency\spaceStep}{\timeShiftOperator}  =  - \timeShiftOperator^3  - \commonOperator (\frequency\spaceStep)\timeShiftOperator^2 + \conjugate{\commonOperator}(\frequency\spaceStep) \timeShiftOperator + 1.
\end{equation*}
We observe that since $\amplificationPolynomial{\frequency\spaceStep}{\timeShiftOperator} = -\amplificationPolynomialConjugate{\frequency\spaceStep}{\timeShiftOperator}$, the polynomial $\amplificationPolynomial{\frequency\spaceStep}{\timeShiftOperator}$ is said to be ``self-inversive'' \cite[Chapter 10]{marden1949geometry}, \cite[Chapter 1]{milovanovic1994topics}, and \cite{vieira2019polynomials}, that is, there exists $\omega$ on the unit circle such that $\amplificationPolynomial{\frequency\spaceStep}{\timeShiftOperator} = \omega \amplificationPolynomialConjugate{\frequency\spaceStep}{\timeShiftOperator}$. 
In our case, $\omega = -1$ is independent of $\frequency\spaceStep$.
From \cite[Theorem 1]{vieira2019polynomials}, since $\amplificationPolynomial{\frequency\spaceStep}{\timeShiftOperator} $ is of odd degree, we deduce that it has at least one root over the unit circle.
Moreover, if $\rootAmplificationPolynomialFourier{}$ is a root of $\amplificationPolynomial{\frequency\spaceStep}{\timeShiftOperator}$, then also $1/\conjugate{\rootAmplificationPolynomialFourier{}}$ is a root of $\amplificationPolynomial{\frequency\spaceStep}{\timeShiftOperator} $.
The zeros of this kind of polynomial either belong to unit circle or occur in pairs conjugate with respect to the unit circle \cite{kim2008zeros}.
Moreover $\timeShiftOperator^{-1}(\amplificationPolynomialConjugate{\frequency\spaceStep}{0} \amplificationPolynomial{\frequency\spaceStep}{\timeShiftOperator} - \amplificationPolynomialConjugate{\frequency\spaceStep}{\timeShiftOperator} \amplificationPolynomial{\frequency\spaceStep}{0}) \equiv 0$, hence---see \cite[Chapter 4]{strikwerda2004finite}---all the roots of the amplification polynomial lie on the unit circle for every wave-number.

\begin{proposition}[Weak instability of \eqref{eq:FourthOrderLeapFrog}]\label{prop:StabilityThirdLeapFrog}
    Assume that the Courant–Friedrichs–Lewy (CFL) condition  $|\courantNumber| < 1/2$ holds.
    Then, \eqref{eq:FourthOrderLeapFrog} is weakly unstable. More precisely:
    \begin{itemize}
        \item For $|\frequency\spaceStep| \in (0, \pi]$, the amplification polynomial $\amplificationPolynomial{\frequency\spaceStep}{\timeShiftOperator}$ given by \eqref{eq:amplificationPolynomialFourthOrderLeapFrog} has distinct roots on the unit circle.
        \item For $|\frequency\spaceStep| = 0$, the amplification polynomial $\amplificationPolynomial{0}{\timeShiftOperator}$ given by \eqref{eq:amplificationPolynomialFourthOrderLeapFrog} has roots $1$ (single) and $-1$ (double).
    \end{itemize}
\end{proposition}
\Cref{prop:StabilityThirdLeapFrog}---proved in \Cref{app:ProofStabilityLeapFrog}---means that \eqref{eq:FourthOrderLeapFrog} is stable for the $\lebesgueSpace{2}$-norm except for the frequency zero, which could cause linear growth of the solution in time.
Indeed, using the language of linear multi-step schemes for ODEs, the scheme is not zero-stable for $\frequency = 0$ and in this case resembles to \cite[Example 12.5 (d)]{suli2003introduction}. 
If we have not had the spatial direction $\spaceVariable$, we could not expect convergence. Still, the presence of the spatial extension helps us in having an overall stable procedure.

One legitimate question concerns the meaning of the CFL condition $|\courantNumber| < 1/2$.
This can be seen by explicitly computing the roots of the amplification polynomial and perform Taylor expansions in the low-frequency limit $|\frequency\spaceStep| \ll 1$.
This is
\begin{multline}\label{eq:expansionConsistencyEigenvalueD1Q3}
    \rootAmplificationPolynomialFourier{1}(\frequency\spaceStep) = 1 - \imaginaryUnit\courantNumber\frequency\spaceStep - \tfrac{1}{2}\courantNumber^2\frequency^2\spaceStep^2 + \tfrac{\imaginaryUnit}{6}\courantNumber^3 \frequency^3\spaceStep^3 + \tfrac{1}{24}\courantNumber^4\frequency^4\spaceStep^4 + \tfrac{\imaginaryUnit \courantNumber}{360}\bigl (5\courantNumber^4 - 10\courantNumber^2 + 2 \bigr )\frequency^5\spaceStep^5 + \bigO{|\frequency\spaceStep|^6} \\
    = e^{-\imaginaryUnit\courantNumber\frequency\spaceStep (1+\bigO{|\frequency\spaceStep|^4})},
\end{multline}
which proves that---as already emphasized---the method is fourth-order accurate. Even more precisely, \confer{} \cite{strikwerda2004finite, coulombel2020neumann}: it exists a constant $C > 0$ such that $\tfrac{1}{\timeStep} |e^{-\imaginaryUnit\courantNumber\frequency\spaceStep} - \rootAmplificationPolynomialFourier{1}(\frequency\spaceStep) | \leq C \spaceStep^4 |\frequency|^5 $ for $|\frequency\spaceStep| \leq \pi$. 
For the parasitic eigenvalues:
\begin{align}
    \rootAmplificationPolynomialFourier{2}(\frequency\spaceStep) &= - e^{\tfrac{\imaginaryUnit\sqrt{3}}{6} (\sqrt{3}\courantNumber + \sqrt{8 - 5\courantNumber^2})\frequency\spaceStep (1+\bigO{|\frequency\spaceStep|^2})}, \label{eq:expansionSecondEigenvalueD1Q3} \\
    \rootAmplificationPolynomialFourier{3}(\frequency\spaceStep) &= - e^{\tfrac{\imaginaryUnit\sqrt{3}}{6} (\sqrt{3}\courantNumber - \sqrt{8 - 5\courantNumber^2})\frequency\spaceStep (1+\bigO{|\frequency\spaceStep|^2})}. \label{eq:expansionThirdEigenvalueD1Q3} 
\end{align}
The exponential form in \eqref{eq:expansionSecondEigenvalueD1Q3} and \eqref{eq:expansionThirdEigenvalueD1Q3} is inspired by \cite[Theorem 19]{mclachlan2002splitting} and emphasizes that the parasitic eigenvalues essentially behave like pseudo-schemes of order $\spaceStep^2$ for a different flow compared to the target equation.
Remark that $8 - 5\courantNumber^2 > 0$ by the constraint $|\courantNumber| \leq 1$ that must naturally hold for an explicit method with spatial stencil of one, see \cite{strang1962trigonometric}.
The first parasitic mode---brought by $\rootAmplificationPolynomialFourier{2}$---propagates backward, whatever the sign of $\courantNumber$, whereas the second mode---carried by $\rootAmplificationPolynomialFourier{3}$---always propagates forward.
Both produce rapid checkerboard-like oscillating solutions since $\rootAmplificationPolynomialFourier{2}(0) = \rootAmplificationPolynomialFourier{3}(0) = -1$.
We would like the parasitic waves to propagate slower than the speed of information of the scheme, which is equal to $\latticeVelocity = \spaceStep / \timeStep$. 
This can be stated as
\begin{equation*}
    \begin{cases}
        \tfrac{\sqrt{3}}{6} (\sqrt{3}\courantNumber + \sqrt{8 - 5\courantNumber^2}) < 1, \qquad &\to \qquad \courantNumber \in (-1, 1/2), \\
        \tfrac{\sqrt{3}}{6} (\sqrt{3}\courantNumber - \sqrt{8 - 5\courantNumber^2}) > -1, \qquad &\to \qquad \courantNumber \in (-1/2, 1),
    \end{cases}
    \qquad \to \qquad |\courantNumber| < 1/2,
\end{equation*}
which is indeed the CFL condition by \Cref{prop:StabilityThirdLeapFrog}.
This demonstrates that in this case, the CFL constraint is a condition on the speed of propagation of information by the parasitic modes.
By studying $\tfrac{\sqrt{3}}{6} (\sqrt{3}\courantNumber + \sqrt{8 - 5\courantNumber^2}) >  \courantNumber$, we see that the velocity of the parasitic waves is always larger than the one of the physical wave, as expected, because otherwise the CFL constraint would have stemmed from the speed of propagation of information by the physical mode.

We now understand why the numerical simulations in \Cref{sec:NumericalExperiment} go against \cite[Theorem 10.6.2]{strikwerda2004finite}.
This results needs the numerical scheme to be genuinely stable\dots and \eqref{eq:FourthOrderLeapFrog} is not---according to \Cref{prop:StabilityThirdLeapFrog}.
The aim of the sections to come is to understand why the numerical scheme still behaves in a stable fashion and converges, though with unexpected rates.

\begin{remark}[A simpler numerical scheme with the similar features]\label{rem:simplerExample}
    We can construct a simpler multi-step scheme with  features analogous to \eqref{eq:FourthOrderLeapFrog}, building an \emph{ad hoc} amplification polynomial. 
    We would like it to have a stable method for all wave-numbers except for the frequency zero, where a double root $-1$ is present.
    We do not request all the roots to be on the unit circle for all $|\frequency\spaceStep| \leq \pi$.
    We thus consider 
    \begin{equation}\label{eq:adHocAmplification}
        \amplificationPolynomial{\frequency\spaceStep}{\timeShiftOperator} = (\timeShiftOperator - \rootAmplificationPolynomialFourierCustom{OS4}(\frequency\spaceStep)) (\timeShiftOperator + \cos(\frequency\spaceStep)) (\timeShiftOperator + 1),
    \end{equation}
    where $\rootAmplificationPolynomialFourierCustom{OS4}(\frequency\spaceStep)$ is the amplification factor of the OS4 scheme. It could indeed be replaced by the one of any dissipative one-step scheme (this constraint would exclude the Lax-Friedrichs scheme since we would have a double root $-1$ at $\frequency\spaceStep = \pi$).
    This type of scheme gives the same surprising results as in \Cref{sec:NumericalExperiment}, which are not included in the paper.
    This is rather paradoxical but instructive: if we had taken the fourth-order one-step scheme associated with $\rootAmplificationPolynomialFourierCustom{OS4}$ and made a (finite) number of iterations at the very beginning with a third-order scheme, we would have preserved an overall fourth-order.
    Similarly, if we had considered parasitic roots in of the amplification polynomial $\rootAmplificationPolynomialFourier{2}$ and $\rootAmplificationPolynomialFourier{3}$ fulfilling the stability condition also at $\frequency\spaceStep = 0$, for example $\amplificationPolynomial{\frequency\spaceStep}{\timeShiftOperator} = (\timeShiftOperator - \rootAmplificationPolynomialFourierCustom{OS4}(\frequency\spaceStep)) (\timeShiftOperator + \tfrac{1}{2}) (\timeShiftOperator + \tfrac{1}{3})$, third-order initializations would have been enough to preserve fourth order.
    These two observations confirm that in this weakly unstable framework, the parasitic roots that we have artificially put along $\rootAmplificationPolynomialFourierCustom{OS4}$ start playing a role as far as consistency (and thus the order of accuracy) is concerned.
\end{remark}

\begin{remark}[Trying to re-establish stability]\label{rem:tryingToStabilize}
    Since, if we do not set $\relaxationParameterSecondMoment = \relaxationParameterThirdMoment = 2$ as we did in \Cref{sec:tuningFreeParam} to enforce fourth-order consistency, we have that $\spectrum{\fourierTransformed{\schemeMatrix}(0)} = \{1, 1-\relaxationParameterSecondMoment, 1-\relaxationParameterThirdMoment \}$, or equivalently $\amplificationPolynomial{0}{ \timeShiftOperator} = \determinant{\timeShiftOperator \identityMatrix - \fourierTransformed{\schemeMatrix}(0)} = (\timeShiftOperator - 1)(\timeShiftOperator + \relaxationParameterSecondMoment - 1)(\timeShiftOperator + \relaxationParameterThirdMoment - 1)$, we could hope to solve the ``collision'' between the second and the third root lying on the unit disk and coinciding by considering $\relaxationParameterSecondMoment = 2$ and $\relaxationParameterThirdMoment = 2 - \spaceStep^{\alpha}$ with $\alpha \in \reals$. Here, $\alpha$ would be chosen large enough not to perturb the fourth order of the scheme.
    This setting does not provide the expected result and gives the same result as \Cref{sec:NumericalExperiment}. 
    This comes from the fact that now \cite[Theorem 4.2.2]{strikwerda2004finite} applies and stability requires that $\rootAmplificationPolynomialFourier{2}(0)$ and $\rootAmplificationPolynomialFourier{3}(0)$ are apart by a positive quantity independent of $\spaceStep$ when the space step is small, \idEst{} $|\rootAmplificationPolynomialFourier{2}(0) - \rootAmplificationPolynomialFourier{3}(0)| \geq c_1$. However, in our case $|\rootAmplificationPolynomialFourier{2}(0) - \rootAmplificationPolynomialFourier{3}(0)| = \spaceStep^{\alpha}$, hence the scheme has not been stabilized.
\end{remark}

\subsection{Understanding convergence}

As observed at the very beginning and through \Cref{rem:simplerExample}, the consistency analysis of the whole scheme does no longer boil down to consider the behavior of the scheme as determined only by $\rootAmplificationPolynomialFourier{1}$, but we have to take all the roots into account and clarify how these different modes are excited and interact by the choice of initialization schemes.

\subsubsection{Several decompositions of the discrete scheme}

This is achieved using several kinds of decompositions of the discrete solution.
Given the amplification polynomial $\amplificationPolynomial{\frequency\spaceStep}{\timeShiftOperator} = \timeShiftOperator^{\numberSteps + 1} + \sum_{\indicesTime = 0}^{\indicesTime = \numberSteps} \coefficientAmplificationPolynomial_{\indicesTime}(\frequency\spaceStep)\timeShiftOperator^{\indicesTime}$ of an explicit scheme, we introduce its companion matrix (or amplification matrix \cite{coulombel2021leray}) 
\begin{equation*}
    \companionMatrix(\frequency\spaceStep) = 
    \begin{bmatrix}
        -\coefficientAmplificationPolynomial_{\numberSteps}(\frequency\spaceStep) & \cdots & -\coefficientAmplificationPolynomial_1(\frequency\spaceStep) & -\coefficientAmplificationPolynomial_0(\frequency\spaceStep) \\
        1 & \cdots & 0 & 0 \\
        \vdots & \ddots & \vdots & \vdots \\
        0 & \cdots & 0 & 0 \\
        0 & \cdots & 1 & 0
    \end{bmatrix}.
\end{equation*}
We form the amplification factors, given, for $\indicesTime \geq \numberSteps + 1$ by 
\begin{equation}\label{eq:definitionAmplificationFactors}
    \amplificationFactorFourier{\indicesTime} (\frequency\spaceStep) = 
    \transpose{\canonicalBasisVector{1}} \companionMatrix(\frequency\spaceStep)^{\indicesTime - \numberSteps} 
    \transpose{(\amplificationFactorFourier{\numberSteps}(\frequency\spaceStep), \dots, \amplificationFactorFourier{1}(\frequency\spaceStep), 1)},
\end{equation}
so that we have $\solutionDiscreteFourier{\indicesTime} (\frequency) = \amplificationFactorFourier{\indicesTime} (\frequency\spaceStep) \solutionDiscreteFourier{0} (\frequency) $.
Here $\amplificationFactorFourier{\numberSteps}(\frequency\spaceStep), \dots, \amplificationFactorFourier{1}(\frequency\spaceStep)$ are the amplification factors of the initialization schemes.
Let us point out a feature concerning the companion matrix of a weakly unstable scheme.
\begin{proposition}\label{prop:GrowthCompanion}
    Let $\companionMatrix(\frequency\spaceStep)$ for $|\frequency\spaceStep| \leq \pi$ be the companion matrix of a weakly unstable Finite Difference scheme.
    Then $\companionMatrix(\tilde{\frequency}\spaceStep)^{\indicesTime}$, for some $|\tilde{\frequency}\spaceStep| \leq \pi$, grows (polynomially) with $\indicesTime \in \naturals$.
\end{proposition} 
\begin{proof}
    Let $|\tilde{\frequency}\spaceStep| \leq \pi$ be one of the frequencies where multiple roots of the amplification polynomial happen to be on the unit circle.
    The companion matrix has coinciding characteristic (\idEst{} the amplification polynomial) and minimal polynomial.
    Therefore, for this frequency, its Jordan canonical form inferred from the minimal polynomial features a block for the multiple eigenvalue on the unit circle having size larger than one, giving the claimed growth.
\end{proof}

\paragraph{Green decomposition}

A tool to understand the different role of the initialization $\amplificationFactorFourier{2}$ compared to $\amplificationFactorFourier{1}$ in \eqref{eq:FourthOrderLeapFrog} can be found using Green functions \cite{cheng1999general, coulombel2022generalized, coulombel22green}.
We remark that the phenomenon observed in \Cref{sec:NumericalExperiment} corresponds to the fact that the minimal requirement to preserve the overall order four is that the vector ${(\amplificationFactorFourier{2}, \amplificationFactorFourier{1}, 1)}$ must to be equal to the eigenvector ${(\rootAmplificationPolynomialFourier{1}^2, \rootAmplificationPolynomialFourier{1}, 1)}$ of the companion matrix $\companionMatrix$ associated with $\rootAmplificationPolynomialFourier{1}$ at different orders for $|\frequency\spaceStep| \ll 1$ according to the component (four for the first and third component, three for the second one).
Otherwise said, one wants at least ${(\amplificationFactorFourier{2}, \amplificationFactorFourier{1}, 1)} = {(\rootAmplificationPolynomialFourier{1}^2 + \bigO{|\frequency\spaceStep|^5}, \rootAmplificationPolynomialFourier{1} + \bigO{|\frequency\spaceStep|^4}, 1 + \bigO{|\frequency\spaceStep|^5})}$.
We therefore isolate the role of each initialization scheme by considering the $\indicesFree$-th Green functions $\greenFunction{\indicesTime}{\indicesFree}(\frequency\spaceStep)$---for $\indicesFree \in \integerInterval{0}{\numberSteps}$---defined by
\begin{equation*}
    \begin{cases}
        \greenFunction{\indicesTime + 1}{\indicesFree} = -\sum\limits_{\indicesFreeTwo = 0}^{\numberSteps} \coefficientAmplificationPolynomial_{\indicesFreeTwo} \greenFunction{\indicesTime + \indicesFreeTwo - \numberSteps}{\indicesFree} = \transpose{\canonicalBasisVector{1}} \companionMatrix \transpose{(\greenFunction{\indicesTime}{\indicesFree}, \dots, \greenFunction{\indicesTime - \numberSteps}{\indicesFree})}, \qquad &\textnormal{for} \quad \indicesTime \geq \numberSteps, \\
        \greenFunction{\indicesTime}{\indicesFree} = \delta_{\indicesTime, \indicesFree}, \qquad &\textnormal{for} \quad \indicesTime \in \integerInterval{0}{\numberSteps}.
    \end{cases}
\end{equation*}
Then, for $\indicesTime \geq \numberSteps + 1$, we have that $\greenFunction{\indicesTime}{\indicesFree} = \transpose{\canonicalBasisVector{1}} \companionMatrix^{\indicesTime - \numberSteps}\canonicalBasisVector{\numberSteps - \indicesFree + 1}$ and, by adding and subtracting well-selected quantities:
\begin{multline}\label{eq:generalGreenDecomposition}
    \amplificationFactorFourier{\indicesTime}(\frequency\spaceStep) = \sum_{\indicesFree = 0}^{\numberSteps} \greenFunction{\indicesTime}{\indicesFree}(\frequency\spaceStep) \amplificationFactorFourier{\indicesFree}(\frequency\spaceStep) = \sum_{\indicesFree = 0}^{\numberSteps} \greenFunction{\indicesTime}{\indicesFree}(\frequency\spaceStep)\rootAmplificationPolynomialFourier{1}(\frequency\spaceStep)^{\indicesFree} + \sum_{\indicesFree = 1}^{\numberSteps} \greenFunction{\indicesTime}{\indicesFree}(\frequency\spaceStep) (\amplificationFactorFourier{\indicesFree}(\frequency\spaceStep) - \rootAmplificationPolynomialFourier{1}(\frequency\spaceStep)^{\indicesFree}) \\
    = \transpose{\canonicalBasisVector{1}} {\companionMatrix}(\frequency\spaceStep)^{\indicesTime - \numberSteps} \transpose{(\rootAmplificationPolynomialFourier{1}(\frequency\spaceStep)^{\numberSteps}, \dots, \rootAmplificationPolynomialFourier{1}(\frequency\spaceStep), 1)} + \sum_{\indicesFree = 1}^{\numberSteps} \greenFunction{\indicesTime}{\indicesFree}(\frequency\spaceStep) (\amplificationFactorFourier{\indicesFree}(\frequency\spaceStep) - \rootAmplificationPolynomialFourier{1}(\frequency\spaceStep)^{\indicesFree}) \\
    = \rootAmplificationPolynomialFourier{1}(\frequency\spaceStep)^{\indicesTime} + \sum_{\indicesFree = 1}^{\numberSteps} \greenFunction{\indicesTime}{\indicesFree}(\frequency\spaceStep) (\amplificationFactorFourier{\indicesFree}(\frequency\spaceStep) - \rootAmplificationPolynomialFourier{1}(\frequency\spaceStep)^{\indicesFree}),
\end{multline}
where the last equality is obtained using the fact that $\transpose{(\rootAmplificationPolynomialFourier{1}(\frequency\spaceStep)^{\numberSteps}, \dots, \rootAmplificationPolynomialFourier{1}(\frequency\spaceStep), 1)}$ is the eigenvector of $\companionMatrix$ relative to the eigenvalue $\rootAmplificationPolynomialFourier{1}$.

\paragraph{Modal decomposition}
Another decomposition of $\amplificationFactorFourier{\indicesTime}$ can be found as follows.
If we assume that for a given wave-number $\frequency$ such that $|\frequency\spaceStep|\leq \pi$, all the roots $\rootAmplificationPolynomialFourier{1}(\frequency\spaceStep), \dots, \rootAmplificationPolynomialFourier{\numberSteps + 1}(\frequency\spaceStep)$ are distinct, we have the so-called ``modal'' decomposition, directly inspired from the theory of linear recurrences, which reads
\begin{equation}\label{eq:generalModalDecomposition}
    \amplificationFactorFourier{\indicesTime}(\frequency\spaceStep) = \sum_{\indicesFree = 1}^{\numberSteps + 1} \coefficientModalDecomposition{\indicesFree}(\frequency\spaceStep) \rootAmplificationPolynomialFourier{\indicesFree}(\frequency\spaceStep)^{\indicesTime},
\end{equation}  
and where the coefficients $\coefficientModalDecomposition{\indicesFree}$ are determined by the initialization schemes $\amplificationFactorFourier{\numberSteps}, \dots, \amplificationFactorFourier{1}$.
In particular, introducing the Vandermonde matrix 
\begin{equation*}
    \vandermondeMatrix(\frequency\spaceStep) = 
    \begin{bmatrix}
        \rootAmplificationPolynomialFourier{1}(\frequency\spaceStep)^{\numberSteps} & \cdots & \rootAmplificationPolynomialFourier{\numberSteps+1}(\frequency\spaceStep)^{\numberSteps} \\
        \vdots & & \vdots \\
        \rootAmplificationPolynomialFourier{1}(\frequency\spaceStep) & \cdots & \rootAmplificationPolynomialFourier{\numberSteps+1}(\frequency\spaceStep) \\
        1 &  & 1
    \end{bmatrix},
\end{equation*}
gives that $\coefficientModalDecomposition{1}, \dots \coefficientModalDecomposition{\numberSteps+1}$ satisfy the linear system $\vandermondeMatrix(\frequency\spaceStep) \transpose{(\coefficientModalDecomposition{1}(\frequency\spaceStep), \dots, \coefficientModalDecomposition{\numberSteps+1}(\frequency\spaceStep))} = \transpose{(\amplificationFactorFourier{\numberSteps}(\frequency\spaceStep), \amplificationFactorFourier{\numberSteps-1}(\frequency\spaceStep), \dots, 1)}$.
The fact that the Vandermonde matrix can be inverted comes from the assumption of dealing with distinct roots.

\paragraph{Comparison between decompositions}

Let us comment on the different properties of the Green decomposition \eqref{eq:generalGreenDecomposition} \emph{vs.} the modal decomposition \eqref{eq:generalModalDecomposition}.
\begin{itemize}
    \item Green decomposition \eqref{eq:generalGreenDecomposition}.
    It focuses on the consistency eigenvalue $\rootAmplificationPolynomialFourier{1}$: one sees the overall scheme after $\indicesTime$ time-steps as the application of the pseudo-scheme associated with $\rootAmplificationPolynomialFourier{1}$ $\indicesTime$-times, plus some perturbation induced by the initialization schemes $\amplificationFactorFourier{\numberSteps}, \dots, \amplificationFactorFourier{1}$ and their deviation from $\rootAmplificationPolynomialFourier{1}^{\numberSteps}, \dots, \rootAmplificationPolynomialFourier{1}$.
    These deviations are weighted by the Green functions.
    This decomposition is therefore used to understand consistency and the fact that the overall scheme reacts differently according to the step fed by a given initialization scheme. 
    Its main drawback is that it mixes the parasitic modes $\rootAmplificationPolynomialFourier{2}, \dots, \rootAmplificationPolynomialFourier{\numberSteps + 1}$ and thus hides their role of rapidly oscillating transport modes at different velocities (see \eqref{eq:expansionSecondEigenvalueD1Q3} and \eqref{eq:expansionThirdEigenvalueD1Q3}) compared to the physical mode.
    \item Modal decomposition \eqref{eq:generalModalDecomposition}. 
    It can informally described by the sentence ``All modes are created equal'', for it does not emphasize any root between $\rootAmplificationPolynomialFourier{1}, \rootAmplificationPolynomialFourier{2}, \dots, \rootAmplificationPolynomialFourier{\numberSteps + 1}$. It aims at alleviating the drawback of the Green decomposition, allowing to see the discrete solution as the superposition of $\numberSteps + 1$ modes, each one having its own velocity.
    This helps the study of qualitative properties of the discrete solution.
    However, this approach is less suitable (at least without explicit computation of the coefficients at hand) to proceed to a rigorous study of the consistency of the scheme.
\end{itemize}

\subsubsection{Stability}\label{sec:stability}

In the cases where the amplification polynomials are given by \eqref{eq:amplificationPolynomialFourthOrderLeapFrog} and \eqref{eq:adHocAmplification}, $\companionMatrix(\frequency\spaceStep)^{\indicesTime - \numberSteps}$ is not uniformly power bounded \cite{leveque1984resolvent}, because the schemes are weakly unstable, \confer{} \Cref{prop:GrowthCompanion}. 
In both cases, we have 
\begin{equation}\label{eq:powersCompanionZeroFrequency}
    \transpose{\canonicalBasisVector{1}}  \companionMatrix(0)^{\indicesTime - 2}  =
    (\tfrac{1}{4}(-1)^{\indicesTime} (2\indicesTime-1)+\tfrac{1}{4}, \tfrac{1}{2}(-1)^{\indicesTime + 1}+\tfrac{1}{2}, \tfrac{1}{4}(-1)^{\indicesTime + 1} (2\indicesTime-3)+\tfrac{1}{4}),
\end{equation}
where the first and last entries diverge linearly with $\indicesTime$.
However, the reason why the simulations remain stable is that this instability is not excited. 
Indeed, any reasonable---meaning at least zero-order accurate---initialization scheme that one can consider is such that $\amplificationFactorFourier{2}(0) = \amplificationFactorFourier{1}(0) = 1$, hence we obtain from \eqref{eq:definitionAmplificationFactors} that $\amplificationFactorFourier{\indicesTime} (0) = 1$ for any $\indicesTime \in \naturals$, which is bounded as $\indicesTime$ grows, and thus keeps the simulation stable, \confer{} \Cref{sec:NumericalExperiment}. 
The other frequencies $\frequency \neq 0$ are stable (\idEst{} $\transpose{\canonicalBasisVector{1}}  \companionMatrix(\frequency\spaceStep)^{\indicesTime - 2}$ is bounded with $\indicesTime$) thanks to \Cref{prop:StabilityThirdLeapFrog}.
This indeed shows that there exists $\constantStabilityGeneralized > 0$ independent of $\indicesTime \in \naturals$ and $|\frequency\spaceStep|\leq \pi$ such that 
\begin{equation}\label{eq:stabilityAmplification}
    |\amplificationFactorFourier{\indicesTime}(\frequency\spaceStep)| \leq \constantStabilityGeneralized.
\end{equation}
Observe that the constant $\constantStabilityGeneralized$ can be---for the considered tests---slightly larger than one, but this is unimportant.
This constant does not depend on the final time $\finalTime$.
Notice that \eqref{eq:stabilityAmplification} exactly coincides---for the considered class of initializations---with the notion of stability stated in \cite[Theorem 2.1.1]{gustafsson1995time}.

\subsubsection{Consistency}

Since initializations must be seriously taken into account due to instabilities, we cannot use the notion of local truncation error to derive an estimate on the global truncation error---the quantity one is eventually interested in. 
We therefore have to propose estimates directly on the global truncation error, as in \cite[Section 2.2]{gustafsson1995time}, using the decomposition on the Green functions \eqref{eq:generalGreenDecomposition}.
\begin{remark}[Smoothness]
    We remark that $\amplificationFactorFourier{\indicesTime}(\frequency\spaceStep)$ and the Green functions $\greenFunction{\indicesTime}{2}(\frequency\spaceStep), \greenFunction{\indicesTime}{1}(\frequency\spaceStep)$, and $\greenFunction{\indicesTime}{0}(\frequency\spaceStep)$ are smooth functions of their argument since they are products of smooth functions of the form $e^{\pm i \frequency\spaceStep}$.
    This is distinct from the difficulties that can arise in the explicit determination of formul\ae{} or of their Taylor expansions as $|\frequency\spaceStep|\ll 1$, due to the fact that the multiplicity of the roots of the amplification polynomial changes with $\frequency\spaceStep$, \confer{} \Cref{prop:StabilityThirdLeapFrog}.
    Since these quantities are smooth functions, their Taylor expansions for $|\frequency \spaceStep| \ll 1$ will be determined using their explicit expressions for the case $|\frequency\spaceStep| \neq 0$, which are found by the standard theory of linear constant coefficient recurrence relations, and then carefully passing to the limit.    
\end{remark}
For every $|\frequency\spaceStep| \leq \pi$ (arguments are sometimes omitted for the sake of compactness), \eqref{eq:generalGreenDecomposition} provides
    \begin{equation}\label{eq:greenDecompositionConsistency}
        \amplificationFactorFourier{\indicesTime}(\frequency\spaceStep) =  \rootAmplificationPolynomialFourier{1}(\frequency\spaceStep)^{\indicesTime}+ \greenFunction{\indicesTime}{2} (\frequency\spaceStep) (\amplificationFactorFourier{2}(\frequency\spaceStep)   - \rootAmplificationPolynomialFourier{1}(\frequency\spaceStep)^2)+ \greenFunction{\indicesTime}{1}(\frequency\spaceStep)  (\amplificationFactorFourier{1}(\frequency\spaceStep)   - \rootAmplificationPolynomialFourier{1}(\frequency\spaceStep) ).
    \end{equation}
    The explicit formul\ae{} in terms of the roots are 
    \begin{equation}\label{eq:explicitSecondGreen}
        \greenFunction{\indicesTime}{2} (\frequency\spaceStep) = 
        \begin{cases}
            \tfrac{\rootAmplificationPolynomialFourier{1}^{\indicesTime}(\rootAmplificationPolynomialFourier{2}-\rootAmplificationPolynomialFourier{3}) - \rootAmplificationPolynomialFourier{2}^{\indicesTime}(\rootAmplificationPolynomialFourier{1}-\rootAmplificationPolynomialFourier{3}) + \rootAmplificationPolynomialFourier{3}^{\indicesTime}(\rootAmplificationPolynomialFourier{1}-\rootAmplificationPolynomialFourier{2})}{\rootAmplificationPolynomialFourier{1}^{2}(\rootAmplificationPolynomialFourier{2}-\rootAmplificationPolynomialFourier{3}) - \rootAmplificationPolynomialFourier{2}^{2}(\rootAmplificationPolynomialFourier{1}-\rootAmplificationPolynomialFourier{3}) + \rootAmplificationPolynomialFourier{3}^{2}(\rootAmplificationPolynomialFourier{1}-\rootAmplificationPolynomialFourier{2})}, \qquad &|\frequency\spaceStep| \neq 0, \\
            \tfrac{1}{4}(2(-1)^{\indicesTime} \indicesTime + (-1)^{\indicesTime + 1} + 1), \qquad &|\frequency\spaceStep| = 0.
        \end{cases}
    \end{equation}
    We observe that this Green function is the potentially explosive one: still it is sufficient to have a zero-order scheme for $\amplificationFactorFourier{2}$ in order to disengage this term. 
    Since $\greenFunction{\indicesTime}{2} (\frequency\spaceStep)$ is smooth, we can obtain higher order terms for the limit $|\frequency\spaceStep| \ll 1$ using \eqref{eq:explicitSecondGreen} for $|\frequency\spaceStep| \neq 0$. This yields 
    \begin{equation*}
        \greenFunction{\indicesTime}{2} (\frequency\spaceStep) = (-1)^{\indicesTime} \left \lfloor \tfrac{\indicesTime}{2} \right \rfloor + (-1)^{\indicesTime} \left \lfloor \tfrac{\indicesTime - 1}{2} \right \rfloor \bigl ( \left \lfloor \tfrac{\indicesTime - 1}{2} \right \rfloor + 1 \bigr ) i\courantNumber\frequency\spaceStep + \bigO{\indicesTime^3}(\frequency\spaceStep)^2 + \bigO{|\frequency\spaceStep|^3}.
    \end{equation*}
    The growth of this Green function with $\indicesTime$ shows why we need---besides stability---one order more for the initialization $\amplificationFactorFourier{2}$, the error coming from this choice is amplified (we might say that it resonates) and accumulates in time analogously to the consistency mode $\rootAmplificationPolynomialFourier{1}^{\indicesTime}$. 
    Thus, it needs to be of the same order as $\rootAmplificationPolynomialFourier{1}$ not to lower the overall order. 
    The explicit form of \eqref{eq:explicitSecondGreen}, in particular the denominator, comes from the fact that for $|\frequency\spaceStep| \neq 0$, the companion matrix is diagonalisable:
    \begin{align*}
        \companionMatrix(\frequency\spaceStep)^{\indicesTime - 2} &= \vandermondeMatrix(\frequency\spaceStep)^{-1} \diagonalMatrix{\rootAmplificationPolynomialFourier{1}(\frequency\spaceStep)^{\indicesTime - 2}, \rootAmplificationPolynomialFourier{2}(\frequency\spaceStep)^{\indicesTime - 2}, \rootAmplificationPolynomialFourier{3}(\frequency\spaceStep)^{\indicesTime - 2}}  \vandermondeMatrix(\frequency\spaceStep) \\
        &= \frac{\adjugate{\vandermondeMatrix(\frequency\spaceStep)}}{\determinant{\vandermondeMatrix(\frequency\spaceStep)}} \diagonalMatrix{\rootAmplificationPolynomialFourier{1}(\frequency\spaceStep)^{\indicesTime - 2}, \rootAmplificationPolynomialFourier{2}(\frequency\spaceStep)^{\indicesTime - 2}, \rootAmplificationPolynomialFourier{3}(\frequency\spaceStep)^{\indicesTime - 2}}  \vandermondeMatrix(\frequency\spaceStep),
    \end{align*}
    with $\determinant{\vandermondeMatrix(\frequency\spaceStep)} = \rootAmplificationPolynomialFourier{1}^{2}(\rootAmplificationPolynomialFourier{2}-\rootAmplificationPolynomialFourier{3}) - \rootAmplificationPolynomialFourier{2}^{2}(\rootAmplificationPolynomialFourier{1}-\rootAmplificationPolynomialFourier{3}) + \rootAmplificationPolynomialFourier{3}^{2}(\rootAmplificationPolynomialFourier{1}-\rootAmplificationPolynomialFourier{2})$.
    Unsurprisingly, $\determinant{\vandermondeMatrix(0)} = 0$. Still, also the numerator $\rootAmplificationPolynomialFourier{1}^{\indicesTime}(\rootAmplificationPolynomialFourier{2}-\rootAmplificationPolynomialFourier{3}) - \rootAmplificationPolynomialFourier{2}^{\indicesTime}(\rootAmplificationPolynomialFourier{1}-\rootAmplificationPolynomialFourier{3}) + \rootAmplificationPolynomialFourier{3}^{\indicesTime}(\rootAmplificationPolynomialFourier{1}-\rootAmplificationPolynomialFourier{2}) = 0$ for $\frequency\spaceStep = 0$.
    For the other Green function:
    \begin{equation*}
        \greenFunction{\indicesTime}{1} (\frequency\spaceStep) = 
        \begin{cases}
            -\tfrac{\rootAmplificationPolynomialFourier{1}^{\indicesTime}(\rootAmplificationPolynomialFourier{2}^2-\rootAmplificationPolynomialFourier{3}^2) - \rootAmplificationPolynomialFourier{2}^{\indicesTime}(\rootAmplificationPolynomialFourier{1}^2-\rootAmplificationPolynomialFourier{3}^2) + \rootAmplificationPolynomialFourier{3}^{\indicesTime}(\rootAmplificationPolynomialFourier{1}^2-\rootAmplificationPolynomialFourier{2}^2)}{\rootAmplificationPolynomialFourier{1}^{2}(\rootAmplificationPolynomialFourier{2}-\rootAmplificationPolynomialFourier{3}) - \rootAmplificationPolynomialFourier{2}^{2}(\rootAmplificationPolynomialFourier{1}-\rootAmplificationPolynomialFourier{3}) + \rootAmplificationPolynomialFourier{3}^{2}(\rootAmplificationPolynomialFourier{1}-\rootAmplificationPolynomialFourier{2})}, \qquad &|\frequency\spaceStep| \neq 0, \\
            -\tfrac{1}{2}((-1)^{\indicesTime} - 1), \qquad &|\frequency\spaceStep| = 0.
        \end{cases}
    \end{equation*}
    As before
    \begin{equation*}
        \greenFunction{\indicesTime}{1} (\frequency\spaceStep) =\tfrac{1-(-1)^{\indicesTime}}{2} + ((-1)^{\indicesTime}-1) \left \lfloor \tfrac{\indicesTime}{2} \right \rfloor i\courantNumber\frequency\spaceStep + \bigO{\indicesTime^2} (\frequency\spaceStep)^2 + \bigO{|\frequency\spaceStep|^3}.
    \end{equation*}
    Remark that the first two terms in this expansions are zero whenever $\indicesTime$ is even.
    This means that the solution at even time steps experiences---in the low frequency limit---a very reduced influence of the first initialization scheme.

    To finish this part, the previous arguments show that there exist constants $C_1, C_2>0$ and $C_3 > 0$ such that, for every $|\frequency\spaceStep| \leq \pi$ and $\indicesTime \geq \numberSteps$
    \begin{equation}\label{eq:modulusGreenFunctionsBound}
        |\greenFunction{\indicesTime}{2}(\frequency\spaceStep)| \leq C_1 \indicesTime + C_2, \qquad |\greenFunction{\indicesTime}{1}(\frequency\spaceStep)| \leq  C_3.
    \end{equation}

\subsubsection{Convergence}

\begin{theorem}[Convergence of \eqref{eq:FourthOrderLeapFrog}]\label{thm:convergenceTheorem}
    Let the CFL condition $|\courantNumber|<1/2$ be satisfied.
    Let $\initialDatum \in \sobolevSpaceEllTwo{\maxOrderAccuracy + 1}(\reals)$, where $\maxOrderAccuracy = 4$ is the order of the bulk scheme and $\maxOrderAccuracy_2, \maxOrderAccuracy_1 \geq 0$ are the orders of accuracy of the second and first initialization schemes.
    Start the discrete scheme using $\solutionDiscrete{0}{} = \evaluationOperator \initialDatum$, where $\evaluationOperator : \lebesgueSpace{2}(\reals) \to \ell^2(\spaceStep\relatives)$ such that $(\evaluationOperator\solution)_{\indicesSpace} = \solution(\gridPointSpace{\indicesSpace})$ for $\indicesSpace \in \relatives$. Then, there exists $C(\gridPointTime{\indicesTime}, \initialDatum, \maxOrderAccuracy, \maxOrderAccuracy_2, \maxOrderAccuracy_1) > 0$ such that, for $\spaceStep$ small enough, the leading-order error estimate for \eqref{eq:FourthOrderLeapFrog} reads:
    \begin{equation*}
        \weightedLTwoNorm{\evaluationOperator \solution(\gridPointTime{\indicesTime}) - \solutionDiscreteVector{\indicesTime}} \leq C(\gridPointTime{\indicesTime}, \initialDatum, \maxOrderAccuracy, \maxOrderAccuracy_2, \maxOrderAccuracy_1) \spaceStep^{\min (\maxOrderAccuracy, \maxOrderAccuracy_2, \maxOrderAccuracy_1 + 1)},
    \end{equation*}
    where 
    \begin{equation*}
        \leadingNormInEstimates = 
        \begin{cases}
            \maxOrderAccuracy_2 + 1, \qquad \text{if} \quad \min(\maxOrderAccuracy, \maxOrderAccuracy_2) < \min(\maxOrderAccuracy, \maxOrderAccuracy_1 ) +1, \quad \maxOrderAccuracy_2 < \maxOrderAccuracy   ,   \qquad &\textnormal{(I)},\\
            \maxOrderAccuracy   + 1, \qquad \text{if} \quad \min(\maxOrderAccuracy, \maxOrderAccuracy_2) < \min(\maxOrderAccuracy, \maxOrderAccuracy_1 ) +1, \quad \maxOrderAccuracy_2 \geq \maxOrderAccuracy,   \qquad &\textnormal{(II)},\\
            \maxOrderAccuracy_2 + 1, \qquad \text{if} \quad \min(\maxOrderAccuracy, \maxOrderAccuracy_2) = \min(\maxOrderAccuracy, \maxOrderAccuracy_1 ) +1, \quad \maxOrderAccuracy_2 < \maxOrderAccuracy,      \qquad &\textnormal{(III)},\\
            \maxOrderAccuracy   + 1, \qquad \text{if} \quad \min(\maxOrderAccuracy, \maxOrderAccuracy_2) = \min(\maxOrderAccuracy, \maxOrderAccuracy_1 ) +1, \quad \maxOrderAccuracy_2 \geq \maxOrderAccuracy,   \qquad &\textnormal{(IV)},\\
            \maxOrderAccuracy_1 + 1, \qquad \text{if} \quad \min(\maxOrderAccuracy, \maxOrderAccuracy_2) > \min(\maxOrderAccuracy, \maxOrderAccuracy_1 ) +1 , \quad \maxOrderAccuracy_1 < \maxOrderAccuracy - 1, \qquad &\textnormal{(V)},       
        \end{cases}
    \end{equation*}
    and
    \begin{equation*}
        C(\gridPointTime{\indicesTime}, \initialDatum, \maxOrderAccuracy, \maxOrderAccuracy_2, \maxOrderAccuracy_1) = 
        \begin{cases}
            (C \gridPointTime{\indicesTime} + C) |\initialDatum|_{\sobolevSpaceEllTwo{\leadingNormInEstimates}}, \qquad &\textnormal{(I)}, \\
            (C \gridPointTime{\indicesTime} + C) |\initialDatum|_{\sobolevSpaceEllTwo{\leadingNormInEstimates}}, \qquad &\textnormal{(II)}, \\
            (C (\gridPointTime{\indicesTime})^2 |\initialDatum|_{\sobolevSpaceEllTwo{\leadingNormInEstimates}}^2 + C\gridPointTime{\indicesTime} |\initialDatum|_{\sobolevSpaceEllTwo{\leadingNormInEstimates-1/2}}^2 + C |\initialDatum|_{\sobolevSpaceEllTwo{\leadingNormInEstimates-1}}^2 )^{1/2}, \qquad &\textnormal{(III)}, \\
            (C (\gridPointTime{\indicesTime})^2 |\initialDatum|_{\sobolevSpaceEllTwo{\leadingNormInEstimates}}^2 + C\gridPointTime{\indicesTime} |\initialDatum|_{\sobolevSpaceEllTwo{\leadingNormInEstimates-1/2}}^2 + C |\initialDatum|_{\sobolevSpaceEllTwo{\leadingNormInEstimates-1}}^2 )^{1/2}, \qquad &\textnormal{(IV)}, \\
            C |\initialDatum|_{\sobolevSpaceEllTwo{\leadingNormInEstimates}}, \qquad &\textnormal{(V)},
        \end{cases}
    \end{equation*}
    and $C>0$ indicates unknown constants that can change at each occurrence.
\end{theorem}
\begin{remark}[On \Cref{thm:convergenceTheorem}]
    We remark the following facts:
    \begin{itemize}
        \item The assumption $\initialDatum \in \sobolevSpaceEllTwo{5}(\reals)$ may be sub-optimal. Especially when the overall order is low due to bad initialization schemes, this order can be observed for initial data which are less than $\sobolevSpaceEllTwo{5}$, \confer{} \Cref{sec:testNonSmoothData}.
        \item The error constants generally trend like $C(\gridPointTime{\indicesTime}, \dots) \sim \gridPointTime{\indicesTime}$ except for (V).
        When they depend on $\gridPointTime{\indicesTime}$, this means that the error accumulates at most linearly in time.
        For (V), the leading-order contribution comes into play at the beginning of the process due to the initialization schemes and cannot be reduced.
    \end{itemize}
\end{remark}
\begin{proof}[Proof of \Cref{thm:convergenceTheorem}]
    The proof is extremely similar to the ones given in \cite[Chapter 10]{strikwerda2004finite}: we provide it for the interested reader.
    Knowing that $\initialDatum \in \sobolevSpaceEllTwo{5}(\reals)$, the initial function and the exact solution are continuous and their point-wise values are well defined at any time.
    Using the triangle inequality
    \begin{equation*}
        \weightedLTwoNorm{\evaluationOperator \solution(\gridPointTime{\indicesTime}) - \solutionDiscreteVector{\indicesTime}} \leq \weightedLTwoNorm{\evaluationOperator \solution(\gridPointTime{\indicesTime}) - \truncationOperator \solution(\gridPointTime{\indicesTime})} + \weightedLTwoNorm{\truncationOperator \solution(\gridPointTime{\indicesTime}) - \solutionDiscreteTruncatedVector{\indicesTime}} + \weightedLTwoNorm{\solutionDiscreteTruncatedVector{\indicesTime} - \solutionDiscreteVector{\indicesTime}},
    \end{equation*}
    where the solution $\solutionDiscreteTruncatedVector{\indicesTime}$ is obtained by applying the multi-step scheme and its initializations on the initial datum $\solutionDiscreteTruncated{0}{\indicesSpace} = (\truncationOperator\initialDatum)(\gridPointSpace{\indicesSpace})$.
    Here, the truncation operator $\truncationOperator : \lebesgueSpace{2}(\reals) \to \ell^2(\spaceStep\relatives)$ is given by $\fourierTransformedLarge{\truncationOperator \solution}(\frequency) = \fourierTransformed { \solution} (\frequency) \indicatorFunction{[0, \pi/\spaceStep]}(|\frequency|)$.
    The interpolation operator $\interpolationOperator :  \ell^2(\spaceStep\relatives) \to \lebesgueSpace{2}(\reals) $ is given by $\fourierTransformedLarge{\interpolationOperator \discrete{\solution}}(\frequency) = \fourierTransformed {\discrete{\solution}} (\frequency) \indicatorFunction{[0, \pi/\spaceStep]}(|\frequency|)$.
    For the first term, by \cite[Theorem 10.1.3]{strikwerda2004finite}
    \begin{equation}\label{eq:firstTerm}
        \weightedLTwoNorm{\evaluationOperator \solution(\gridPointTime{\indicesTime}) - \truncationOperator \solution(\gridPointTime{\indicesTime})} \leq C\spaceStep^{5} |\initialDatum|_{\sobolevSpaceEllTwo{5}}.
    \end{equation}
    For the last term, we have 
    \begin{multline}\label{eq:lastTerm}
        \weightedLTwoNorm{\solutionDiscreteTruncatedVector{\indicesTime} - \solutionDiscreteVector{\indicesTime}}^2 = \int_{|\frequency\spaceStep| \leq \pi} |\amplificationFactorFourier{\indicesTime}(\frequency\spaceStep) (\fourierTransformedLarge{\truncationOperator\initialDatum}(\frequency) - \fourierTransformedLarge{\evaluationOperator\initialDatum}(\frequency))|^2 \differential{\frequency} \leq \constantStabilityGeneralized^2 \int_{|\frequency\spaceStep| \leq \pi} |\fourierTransformedLarge{\truncationOperator\initialDatum}(\frequency) - \fourierTransformedLarge{\evaluationOperator\initialDatum}(\frequency)|^2 \differential{\frequency} \\
        \leq \constantStabilityGeneralized^2 C \spaceStep^{5} |\initialDatum|_{\sobolevSpaceEllTwo{5}}^2,
    \end{multline}
    where the stability \eqref{eq:stabilityAmplification} and \cite[Theorem 10.1.3]{strikwerda2004finite} have been used.
    For the central term:
    \begin{multline*}
        \weightedLTwoNorm{\truncationOperator \solution(\gridPointTime{\indicesTime}) - \solutionDiscreteTruncatedVector{\indicesTime}}^2 = \int_{|\frequency\spaceStep| \leq \pi} |\fourierTransformed{\solution}(\gridPointTime{\indicesTime}, \frequency) - \solutionDiscreteTruncatedFourier{\indicesTime}(\frequency)|^2 \differential{\frequency} = \int_{\reals} |\fourierTransformed{\solution}(\gridPointTime{\indicesTime}, \frequency) - \fourierTransformedLarge{\interpolationOperator \solutionDiscreteTruncated{\indicesTime}{}}(\frequency)|^2 \differential{\frequency} - \int_{|\frequency\spaceStep| > \pi} |\fourierTransformed{\solution}(\gridPointTime{\indicesTime}, \frequency)|^2\differential{\frequency} \\
        \leq \lVert \solution(\gridPointTime{\indicesTime}, \cdot) - \interpolationOperator \solutionDiscreteTruncated{\indicesTime}{}\rVert_{\lebesgueSpace{2}(\reals)}^2 = \int_{|\frequency\spaceStep| \leq \pi} |e^{-i\indicesTime\courantNumber\frequency\spaceStep} - \amplificationFactorFourier{\indicesTime}(\frequency\spaceStep)|^2 |{\initialDatumFourier}(\frequency)|^2 \differential{\frequency} + \int_{|\frequency\spaceStep| > \pi} |e^{-i\indicesTime\courantNumber\frequency\spaceStep} |^2 |{\initialDatumFourier}(\frequency)|^2 \differential{\frequency},
    \end{multline*}
    thanks to the Parseval's identity.
    We have 
    \begin{multline}\label{eq:totalErrorInequality}
        |e^{-\imaginaryUnit \indicesTime\courantNumber\frequency\spaceStep} - \amplificationFactorFourier{\indicesTime}| \leq |e^{-\imaginaryUnit \indicesTime\courantNumber\frequency\spaceStep} - \rootAmplificationPolynomialFourier{1}^{\indicesTime}| + |\greenFunction{\indicesTime}{2}| |\amplificationFactorFourier{2} - \rootAmplificationPolynomialFourier{1}^2|  + |\greenFunction{\indicesTime}{1}| |\amplificationFactorFourier{1} - \rootAmplificationPolynomialFourier{1}| \\
        \leq \indicesTime |e^{-\imaginaryUnit \courantNumber\frequency\spaceStep} - \rootAmplificationPolynomialFourier{1}| + (C_1 \indicesTime + C_2) |\amplificationFactorFourier{2} - \rootAmplificationPolynomialFourier{1}^2| + C_3  |\amplificationFactorFourier{1} - \rootAmplificationPolynomialFourier{1}|\\
        \leq C_4 \indicesTime |\frequency\spaceStep|^{\maxOrderAccuracy + 1} + C_5 (C_1 \indicesTime + C_2) |\frequency\spaceStep|^{\min(\maxOrderAccuracy, \maxOrderAccuracy_2) + 1} + C_3 C_6  |\frequency\spaceStep|^{\min(\maxOrderAccuracy, \maxOrderAccuracy_1) + 1} \\
        = \underbrace{\tfrac{C_5}{\latticeVelocity} \gridPointTime{\indicesTime} \spaceStep^{\maxOrderAccuracy} |\frequency|^{\maxOrderAccuracy + 1}}_{\textnormal{bulk scheme}} + \underbrace{C_6 (C_2 \tfrac{\gridPointTime{\indicesTime}}{\latticeVelocity} \spaceStep^{\min(\maxOrderAccuracy, \maxOrderAccuracy_2)} + C_3 \spaceStep^{\min(\maxOrderAccuracy, \maxOrderAccuracy_2) + 1}) |\frequency|^{\min(\maxOrderAccuracy, \maxOrderAccuracy_2) + 1}}_{\textnormal{2nd initialization scheme}} + \underbrace{C_4 C_7 \spaceStep^{\min(\maxOrderAccuracy, \maxOrderAccuracy_1) + 1} |\frequency|^{\min(\maxOrderAccuracy, \maxOrderAccuracy_1) + 1}}_{\textnormal{1st initialization scheme}}.
    \end{multline}
    The first inequality comes from the triangle inequality applied using \eqref{eq:greenDecompositionConsistency}. The second one uses \cite[Equation (10.1.7)]{strikwerda2004finite} and the found dependence of the Green functions in $\indicesTime$, see \eqref{eq:modulusGreenFunctionsBound}.
    The third inequality comes from the order of the schemes, \confer{} \cite{coulombel2020neumann}.
    When taking the square of the previous inequality back into the integral, all the stemming Sobolev semi-norms exist thanks to the smoothness assumption on the initial datum.
    It is therefore time to let $\spaceStep \to 0$ and identify which term is the leading order term. This can be done directly on the global truncation error term $|e^{-\imaginaryUnit \indicesTime\courantNumber\frequency\spaceStep} - \amplificationFactorFourier{\indicesTime}|$.
    We distinguish all the cases 
    \begin{itemize}
        \item $\min(\maxOrderAccuracy, \maxOrderAccuracy_2) < \min(\maxOrderAccuracy, \maxOrderAccuracy_1) + 1$: the second initialization scheme is the limiting one. 
        \begin{itemize}
            \item[(I)] $\min(\maxOrderAccuracy, \maxOrderAccuracy_2) < \maxOrderAccuracy$, equivalently $\maxOrderAccuracy_2 < \maxOrderAccuracy$. In this case, which covers $\testCaseId{1, 1}$, the leading order term in terms of $\spaceStep$ will be 
            \begin{equation*}
                |e^{-\imaginaryUnit \indicesTime\courantNumber\frequency\spaceStep} - \amplificationFactorFourier{\indicesTime}| \leq C_2 C_6  \tfrac{\gridPointTime{\indicesTime}}{\latticeVelocity} \spaceStep^{\min(\maxOrderAccuracy, \maxOrderAccuracy_2)} |\frequency|^{\min(\maxOrderAccuracy, \maxOrderAccuracy_2) + 1} = C_2 C_6  \tfrac{\gridPointTime{\indicesTime}}{\latticeVelocity} \spaceStep^{\maxOrderAccuracy_2} |\frequency|^{\maxOrderAccuracy_2 + 1}.
            \end{equation*}
            \item[(II)] $\min(\maxOrderAccuracy, \maxOrderAccuracy_2) = \maxOrderAccuracy$, equivalently $\maxOrderAccuracy_2 \geq \maxOrderAccuracy$, covering the case $\testCaseId{4, 4}$:
            \begin{equation*}
                |e^{-\imaginaryUnit \indicesTime\courantNumber\frequency\spaceStep} - \amplificationFactorFourier{\indicesTime}| \leq 
                \tfrac{C_5 + C_2C_6 }{\latticeVelocity} \gridPointTime{\indicesTime} \spaceStep^{\maxOrderAccuracy} |\frequency|^{\maxOrderAccuracy + 1}.
            \end{equation*}
        \end{itemize}
        \item $\min(\maxOrderAccuracy, \maxOrderAccuracy_2) = \min(\maxOrderAccuracy, \maxOrderAccuracy_1 ) +1$: both initialization scheme contribute equally. 
        \begin{itemize}
            \item[(III)] $\min(\maxOrderAccuracy, \maxOrderAccuracy_2) = \min(\maxOrderAccuracy, \maxOrderAccuracy_1)+1 < \maxOrderAccuracy$, \idEst{} $\maxOrderAccuracy_2 < \maxOrderAccuracy$ and $\maxOrderAccuracy_1 < \maxOrderAccuracy - 1$ (this latter condition is redundant), which covers $\testCaseId{1, 2}$. We have
            \begin{align*}
                |e^{-\imaginaryUnit \indicesTime\courantNumber\frequency\spaceStep} - \amplificationFactorFourier{\indicesTime}| &\leq 
               (C_2 C_6 \tfrac{\gridPointTime{\indicesTime}}{\latticeVelocity}  |\frequency|^{\min(\maxOrderAccuracy, \maxOrderAccuracy_2) + 1} + C_4 C_7 |\frequency|^{\min(\maxOrderAccuracy, \maxOrderAccuracy_2)})\spaceStep^{\min(\maxOrderAccuracy, \maxOrderAccuracy_2)} \\
               &= (C_2 C_6 \tfrac{\gridPointTime{\indicesTime}}{\latticeVelocity}  |\frequency|^{\maxOrderAccuracy_2 + 1} + C_4 C_7 |\frequency|^{\maxOrderAccuracy_2})\spaceStep^{\maxOrderAccuracy_2}.
            \end{align*}
            \item[(IV)] $\min(\maxOrderAccuracy, \maxOrderAccuracy_2) = \min(\maxOrderAccuracy, \maxOrderAccuracy_1)+1 = \maxOrderAccuracy$, \idEst{} $\maxOrderAccuracy_2 \geq \maxOrderAccuracy$ and $\maxOrderAccuracy_1 \geq \maxOrderAccuracy - 1$ (this latter condition is redundant), which covers $\testCaseId{3, 4}$. We have 
            \begin{equation*}
                |e^{-\imaginaryUnit \indicesTime\courantNumber\frequency\spaceStep} - \amplificationFactorFourier{\indicesTime}| \leq (\tfrac{C_5 + C_2 C_6}{\latticeVelocity} \gridPointTime{\indicesTime}  |\frequency|^{\maxOrderAccuracy + 1}  + C_4 C_7 |\frequency|^{\maxOrderAccuracy} )\spaceStep^{\maxOrderAccuracy}.
            \end{equation*}
        \end{itemize}
        \item $\min(\maxOrderAccuracy, \maxOrderAccuracy_2) > \min(\maxOrderAccuracy, \maxOrderAccuracy_1 ) +1$: the first initialization scheme is the limiting one.
        The only possible case is $\min(\maxOrderAccuracy, \maxOrderAccuracy_1) + 1 < \maxOrderAccuracy$, equivalently $\maxOrderAccuracy_1 < \maxOrderAccuracy - 1$, indicated by (V). In this case, which covers $\testCaseId{1, 3}$:
        \begin{equation*}
            |e^{-\imaginaryUnit \indicesTime\courantNumber\frequency\spaceStep} - \amplificationFactorFourier{\indicesTime}| \leq C_4 C_7 \spaceStep^{\min(\maxOrderAccuracy,\maxOrderAccuracy_1) + 1} |\frequency|^{\min(\maxOrderAccuracy, \maxOrderAccuracy_1) + 1} = C_4 C_7 \spaceStep^{\maxOrderAccuracy_1 + 1} |\frequency|^{\maxOrderAccuracy_1 + 1}.
        \end{equation*}
    \end{itemize}
    This shows that the order in $\spaceStep$ is given by $\spaceStep^{\min(\maxOrderAccuracy, \maxOrderAccuracy_2, \maxOrderAccuracy_1 + 1)}$.
    We introduce 
    \begin{equation*}
        \leadingNormInEstimates = 
        \begin{cases}
            \maxOrderAccuracy_2 + 1, \qquad \text{if} \quad \min(\maxOrderAccuracy, \maxOrderAccuracy_2) < \min(\maxOrderAccuracy, \maxOrderAccuracy_1 ) +1, \quad \maxOrderAccuracy_2 < \maxOrderAccuracy   ,   \qquad &\textnormal{(I)}, \\
            \maxOrderAccuracy   + 1, \qquad \text{if} \quad \min(\maxOrderAccuracy, \maxOrderAccuracy_2) < \min(\maxOrderAccuracy, \maxOrderAccuracy_1 ) +1, \quad \maxOrderAccuracy_2 \geq \maxOrderAccuracy,   \qquad &\textnormal{(II)}, \\
            \maxOrderAccuracy_2 + 1, \qquad \text{if} \quad \min(\maxOrderAccuracy, \maxOrderAccuracy_2) = \min(\maxOrderAccuracy, \maxOrderAccuracy_1 ) +1, \quad \maxOrderAccuracy_2 < \maxOrderAccuracy,      \qquad &\textnormal{(III)}, \\
            \maxOrderAccuracy   + 1, \qquad \text{if} \quad \min(\maxOrderAccuracy, \maxOrderAccuracy_2) = \min(\maxOrderAccuracy, \maxOrderAccuracy_1 ) +1, \quad \maxOrderAccuracy_2 \geq \maxOrderAccuracy,   \qquad &\textnormal{(IV)}, \\
            \maxOrderAccuracy_1 + 1, \qquad \text{if} \quad \min(\maxOrderAccuracy, \maxOrderAccuracy_2) > \min(\maxOrderAccuracy, \maxOrderAccuracy_1 ) +1 , \quad \maxOrderAccuracy_1 < \maxOrderAccuracy - 1, \qquad &\textnormal{(V)}.   
        \end{cases}
    \end{equation*}
    Observe that $\leadingNormInEstimates \geq \min(\maxOrderAccuracy, \maxOrderAccuracy_2, \maxOrderAccuracy_1 + 1)$.
    By \cite[Equation (10.1.10)]{strikwerda2004finite}, we have that $\int_{|\frequency\spaceStep| > \pi} |e^{-i\indicesTime\courantNumber\frequency\spaceStep} |^2 |{\initialDatumFourier}(\frequency)|^2 \differential{\frequency} \leq C \spaceStep^{5} |\initialDatum|_{\sobolevSpaceEllTwo{5}}^2$.
    We gain
    \begin{equation*}
        \weightedLTwoNorm{\truncationOperator \solution(\gridPointTime{\indicesTime}) - \solutionDiscreteTruncatedVector{\indicesTime}} \leq C(\gridPointTime{\indicesTime}, \initialDatum, \maxOrderAccuracy, \maxOrderAccuracy_2, \maxOrderAccuracy_1) \spaceStep^{\min(\maxOrderAccuracy, \maxOrderAccuracy_2, \maxOrderAccuracy_1 + 1)},
    \end{equation*}
    where, by indicating all the constants by $C>0$ (each one is different):
    \begin{equation*}
        C(\gridPointTime{\indicesTime}, \initialDatum, \maxOrderAccuracy, \maxOrderAccuracy_2, \maxOrderAccuracy_1) = 
        \begin{cases}
            C \gridPointTime{\indicesTime} |\initialDatum|_{\sobolevSpaceEllTwo{\leadingNormInEstimates}}, \qquad &\textnormal{(I)}, \\
            C \gridPointTime{\indicesTime} |\initialDatum|_{\sobolevSpaceEllTwo{\leadingNormInEstimates}}, \qquad &\textnormal{(II)}, \\
            (C (\gridPointTime{\indicesTime})^2 |\initialDatum|_{\sobolevSpaceEllTwo{\leadingNormInEstimates}}^2 + C\gridPointTime{\indicesTime} |\initialDatum|_{\sobolevSpaceEllTwo{\leadingNormInEstimates-1/2}}^2 + C |\initialDatum|_{\sobolevSpaceEllTwo{\leadingNormInEstimates-1}}^2 )^{1/2}, \qquad &\textnormal{(III)}, \\
            (C (\gridPointTime{\indicesTime})^2 |\initialDatum|_{\sobolevSpaceEllTwo{\leadingNormInEstimates}}^2 + C\gridPointTime{\indicesTime} |\initialDatum|_{\sobolevSpaceEllTwo{\leadingNormInEstimates-1/2}}^2 + C |\initialDatum|_{\sobolevSpaceEllTwo{\leadingNormInEstimates-1}}^2 )^{1/2}, \qquad &\textnormal{(IV)}, \\
            C |\initialDatum|_{\sobolevSpaceEllTwo{\leadingNormInEstimates}}, \qquad &\textnormal{(V)}.
        \end{cases}
    \end{equation*}
    The overall claim comes adding \eqref{eq:firstTerm} and \eqref{eq:lastTerm}, which are nevertheless negligible.
\end{proof}

\subsection{Error behavior in time}

The behavior of the error of actual numerical computations as time goes on can be studied using the modal decomposition \eqref{eq:generalModalDecomposition}: for every $|\frequency\spaceStep| \leq \pi$, we have that 
    \begin{equation}\label{eq:modalExpansion}
        \amplificationFactorFourier{\indicesTime}(\frequency\spaceStep) =
        \begin{cases}
            \coefficientModalDecomposition{1}(\frequency\spaceStep) \rootAmplificationPolynomialFourier{1}(\frequency\spaceStep)^{\indicesTime} + \coefficientModalDecomposition{2}(\frequency\spaceStep) \rootAmplificationPolynomialFourier{2}(\frequency\spaceStep)^{\indicesTime} + \coefficientModalDecomposition{3}(\frequency\spaceStep) \rootAmplificationPolynomialFourier{3}(\frequency\spaceStep)^{\indicesTime}, \qquad &|\frequency\spaceStep| \neq 0, \\
            \coefficientModalDecompositionDegenerate{1} + \coefficientModalDecompositionDegenerate{2}(-1)^{\indicesTime} + \coefficientModalDecompositionDegenerate{3}(-1)^{\indicesTime}\indicesTime , \qquad &|\frequency\spaceStep| = 0.
        \end{cases}
    \end{equation}
    Here, the coefficients $\coefficientModalDecomposition{1}, \coefficientModalDecomposition{2}, \coefficientModalDecomposition{3}$, and $\coefficientModalDecompositionDegenerate{1}$, $\coefficientModalDecompositionDegenerate{2}$, $\coefficientModalDecompositionDegenerate{3} \in \reals$ are determined using the initialization schemes $\amplificationFactorFourier{2}$ and $\amplificationFactorFourier{1}$. For the initialization schemes that we have considered, $\coefficientModalDecompositionDegenerate{1} = 1$ and $\coefficientModalDecompositionDegenerate{2} = \coefficientModalDecompositionDegenerate{3} = 0$, \confer{} \eqref{eq:stabilityAmplification}.
    We remind that, contrarily to what this explicit formula \eqref{eq:modalExpansion} suggests, $\amplificationFactorFourier{\indicesTime}(\frequency\spaceStep)$ is a smooth function for every $|\frequency\spaceStep| \leq \pi$.
    In order to find its Taylor expansion for $|\frequency\spaceStep| \ll 1$, we can rely on its explicit representation for $|\frequency\spaceStep| \neq 0$. 
    This gives the system, for $|\frequency\spaceStep| \neq 0$: $\vandermondeMatrix(\frequency\spaceStep) \transpose{(\coefficientModalDecomposition{1}(\frequency\spaceStep), \coefficientModalDecomposition{2}(\frequency\spaceStep), \coefficientModalDecomposition{3}(\frequency\spaceStep))} = \transpose{(\amplificationFactorFourier{2}(\frequency\spaceStep), \amplificationFactorFourier{1}(\frequency\spaceStep), 1)}$.
    We can inverse the Vandermonde matrix to give  $\coefficientModalDecomposition{1}$, $\coefficientModalDecomposition{2}$, and $\coefficientModalDecomposition{3}$.
    Let us study some given choice of initialization scheme.

    \begin{figure}[h]
        \begin{center}
            \begin{footnotesize}
                \definecolor{colorOne}{RGB}{88, 7, 115}
                \definecolor{colorTwo}{RGB}{225, 99, 14}
                \definecolor{colorThree}{RGB}{217, 13, 67}
                \def\svgwidth{0.8\textwidth}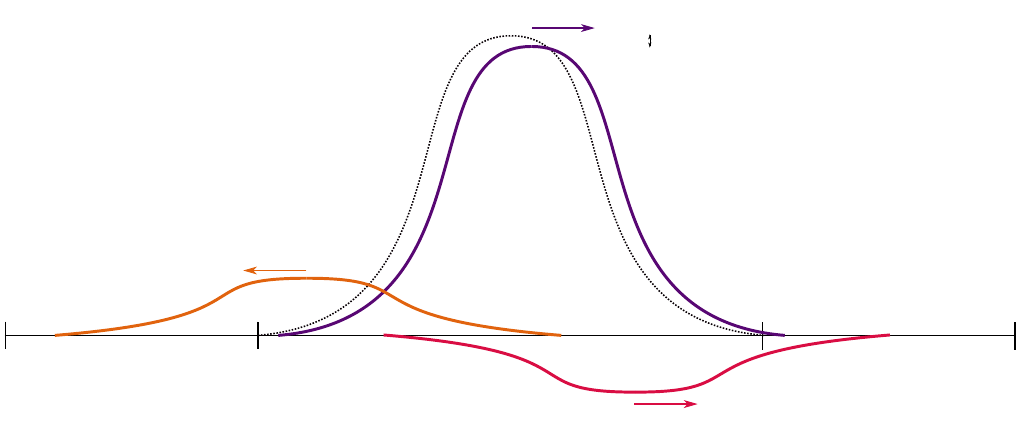
            \end{footnotesize}
        \end{center}\caption{\label{fig:packets}Packets propagated by different modes for the case $\testCaseId{1, 2}$. For the ones carried by $\rootAmplificationPolynomialFourier{2}$ and $\rootAmplificationPolynomialFourier{3}$, we draw an envelope-like shape to highlight the rapidly oscillating nature between odd and even time-steps.}
    \end{figure}

\begin{itemize}
    \item Initialization $\testCaseId{1, 2}$. We obtain, in the limit $|\frequency\spaceStep| \ll 1$, $\coefficientModalDecomposition{1}(\frequency\spaceStep) = 1 + \bigO{|\frequency\spaceStep|^2}$, $\coefficientModalDecomposition{2}(\frequency\spaceStep) = \bigO{|\frequency\spaceStep|}$, and $\coefficientModalDecomposition{3}(\frequency\spaceStep) = \bigO{|\frequency\spaceStep|}$.
    This results in the decomposition of the discrete solution---see also \Cref{fig:packets}---as
    \begin{multline*}
        \amplificationFactorFourier{\indicesTime}(\frequency\spaceStep) = (1 + \bigO{|\frequency\spaceStep|^2}) e^{-\imaginaryUnit \advectionVelocity\gridPointTime{\indicesTime}\frequency(1+\bigO{|\frequency\spaceStep|^4})} \\
        + \Bigl ( \frac{\imaginaryUnit\sqrt{3}(1-\courantNumber^2)}{2\sqrt{8-5\courantNumber^2}} \frequency\spaceStep  + \bigO{|\frequency\spaceStep|^2}\Bigr ) (-1)^{\indicesTime}  e^{\tfrac{\imaginaryUnit\sqrt{3} \gridPointTime{\indicesTime}}{6} (\sqrt{3}\advectionVelocity + \sqrt{8\latticeVelocity^2 - 5\advectionVelocity^2})\frequency (1+\bigO{|\frequency\spaceStep|^2})} \\
        + \Bigl ( -\frac{\imaginaryUnit\sqrt{3}(1-\courantNumber^2)}{2\sqrt{8-5\courantNumber^2}} \frequency\spaceStep  + \bigO{|\frequency\spaceStep|^2}\Bigr ) (-1)^{\indicesTime}  e^{\tfrac{\imaginaryUnit\sqrt{3} \gridPointTime{\indicesTime}}{6} (\sqrt{3}\advectionVelocity - \sqrt{8\latticeVelocity^2 - 5\advectionVelocity^2})\frequency (1+\bigO{|\frequency\spaceStep|^2})}.
    \end{multline*}
    The physical mode---which is accurate at order four---is present with a distortion of order two, see first row. 
    What lowers the overall order to one are the rapidly oscillating spurious modes which have amplitude $\bigO{|\frequency\spaceStep|}$ (second and third rows).

    \begin{figure}[h]
        \begin{center}
            \includegraphics[scale = 0.99]{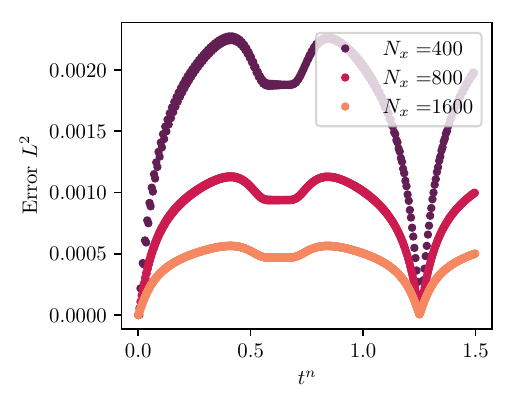}
            \includegraphics[scale = 0.99]{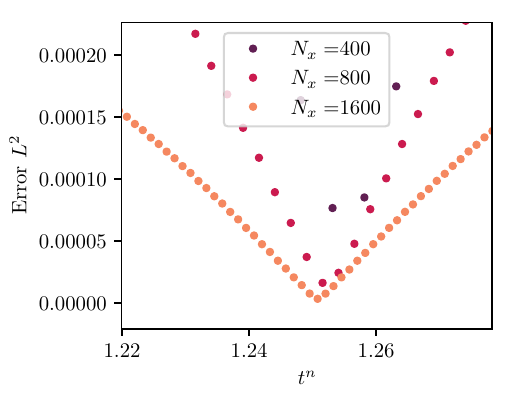}
        \end{center}\caption{\label{fig:21_error_time}$L^2$ error in time for $\testCaseId{1, 2}$ using $\numberSpacePoint$ grid points. A zoom around $1.25$ is proposed on the right.}
    \end{figure}

    \begin{figure}[h]
        \begin{center}
            \includegraphics{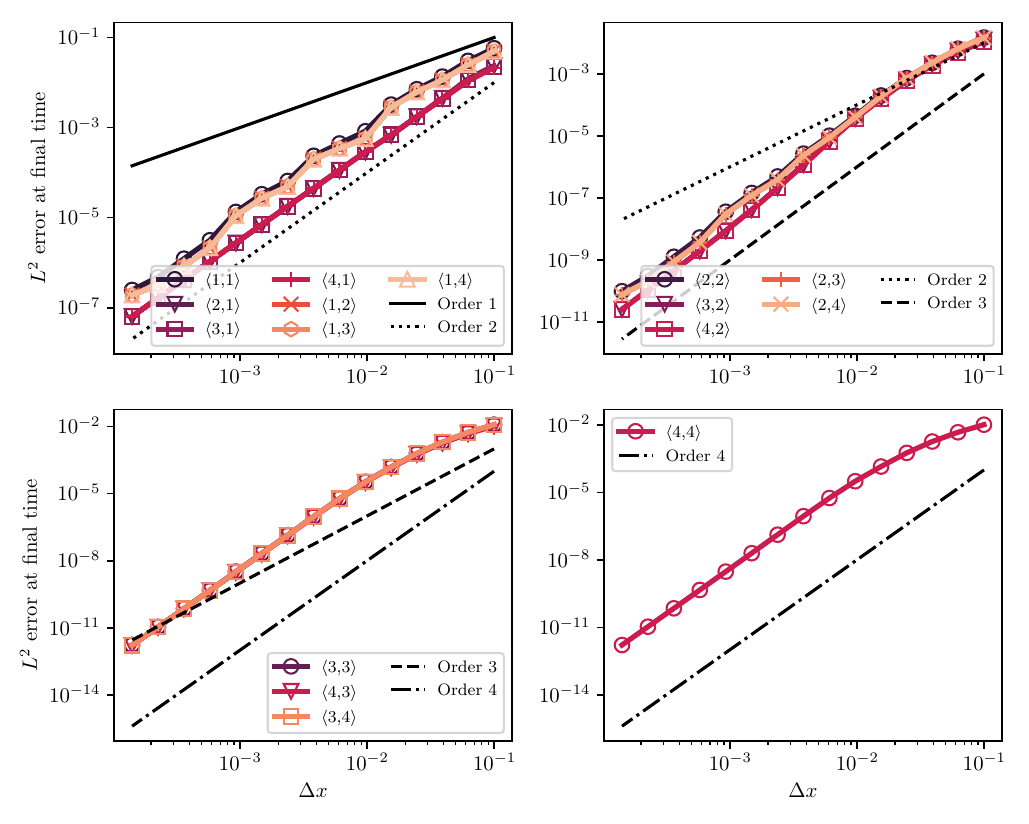}
        \end{center}\caption{\label{fig:convergenceTf125}Error for \eqref{eq:FourthOrderLeapFrog} at final time $\finalTime = \tfrac{2}{|\velocityMode{2}| + |\velocityMode{3}|} \approx 1.25$ with different initialization schemes.}
    \end{figure}

    Looking at the error in time, see \Cref{fig:21_error_time}, we see that for fixed times, the error is proportional to $\spaceStep$. However, for a specific time around $\gridPointTime{\indicesTime} = 1.25$, we see that the error seems to be practically zero. Looking at the magnification in this area, we see that here convergence seems quadratic.
    We explain this spectacular decrease of the error---and enhanced convergence rate---by a destructive interference between the mode brought by $\rootAmplificationPolynomialFourier{2}$ and the one by $\rootAmplificationPolynomialFourier{3}$, which are those carrying the $\bigO{\spaceStep}$ part of the error, see \Cref{fig:packets}.
    At $\courantNumber=1/4$, the dimensionless velocities for each mode are 
    \begin{align*}
        \rootAmplificationPolynomialFourier{1} \quad \leftrightarrow \quad \velocityMode{1} = \courantNumber = 0.25, \qquad &\rootAmplificationPolynomialFourier{2} \quad \leftrightarrow \quad \velocityMode{2} = -\tfrac{\sqrt{3}}{6} (\sqrt{3}\courantNumber + \sqrt{8 - 5\courantNumber^2}) \approx-0.93, \\
        &\rootAmplificationPolynomialFourier{3} \quad \leftrightarrow \quad \velocityMode{3} = -\tfrac{\sqrt{3}}{6} (\sqrt{3}\courantNumber - \sqrt{8 - 5\courantNumber^2}) \approx 0.68. 
    \end{align*}
    The initial envelope has support $\support{\initialDatum} = [-1/2, 1/2]$. We consider that the maximal interference between this envelope transported by the mode $\rootAmplificationPolynomialFourier{2}$ and the one transported by $\rootAmplificationPolynomialFourier{3}$ takes place when the peaks, located at $\spaceVariable = 0$ at $\timeVariable = 0$, meet again due to the periodic boundary conditions. In this case, the two symmetric packets coincide.
    The first peak to reach the boundary ($\spaceVariable = -1$) is the one transported by $\rootAmplificationPolynomialFourier{2}$, since $|\velocityMode{2}| > |\velocityMode{3}|$. 
    The peaks met again at time
    \begin{equation*}
         \frac{2}{|\velocityMode{2}| + |\velocityMode{3}|} \approx 1.25,
    \end{equation*}
    which is---unsurprisingly---the one where the error had its minimum in time.
    If we repeat the convergence test for different initializations selecting precisely $\finalTime = 1.25$, we obtain the result in \Cref{fig:convergenceTf125}.
    Now we observe $\textnormal{overall order} = \min (\maximumOrder, \maximumOrder_2 + 1, \maximumOrder_1 + 1)$, as in the genuinely stable framework, thanks to the fact that the parasitic modes, carrying the dominant part of the error, cancel out thanks to periodicity. 
    When the packets are located at the same place thanks to periodicity, the sum of the modes yields a term $\bigO{|\frequency\spaceStep|^2} \solutionDiscreteFourier{0}(\frequency)$, which is second-order in $\spaceStep$, giving the destructive interference.
    In order to interpret the time behavior of the error more closely, we observe that the time where the leftmost point of $\support{\initialDatum}$, namely $-1/2$, transported by $\rootAmplificationPolynomialFourier{2}$,  and the rightmost point of $\support{\initialDatum}$, namely $1/2$, transported by $\rootAmplificationPolynomialFourier{3}$ merge is at time $\timeVariable = 1/(|\velocityMode{2}| + |\velocityMode{3}|) \approx 0.62$, which is another remarkable time on \Cref{fig:21_error_time}.
    In this figure, the articulate pattern of the error is made up of the interactions of the different waves/modes sustained by the numerical scheme also due to the periodic boundary conditions.

    By writing the global truncation error coefficient in the low-frequency limit:
    \begin{equation*}
        |e^{-i\indicesTime\courantNumber\frequency\spaceStep} - \amplificationFactorFourier{\indicesTime}(\frequency\spaceStep)|^2 = 
        \begin{cases}
            \bigl ( \tfrac{\indicesTime}{2}\bigr )^2 (\courantNumber^2 - 1)^2 (\frequency\spaceStep)^4 + \bigO{|\frequency\spaceStep|^6}, \qquad &\indicesTime~\text{even}, \\
            \bigl ( \tfrac{\indicesTime - 1}{2}\bigr )^2 (\courantNumber^2 - 1)^2 (\frequency\spaceStep)^4 + \bigO{|\frequency\spaceStep|^6}, \qquad &\indicesTime~\text{odd}. 
        \end{cases}
    \end{equation*}
    we see that even and odd steps behave essentially in the same way, as shown in \Cref{fig:21_error_time}. The fixed $-1/2$ term is what remains of the initialization schemes.

    \begin{figure}[h]
        \begin{center}
            \includegraphics[scale = 0.99]{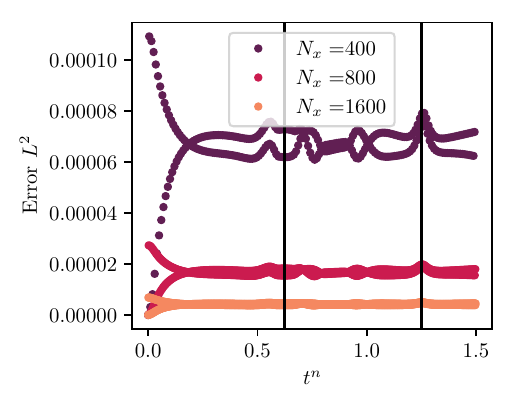}
            \includegraphics[scale = 0.99]{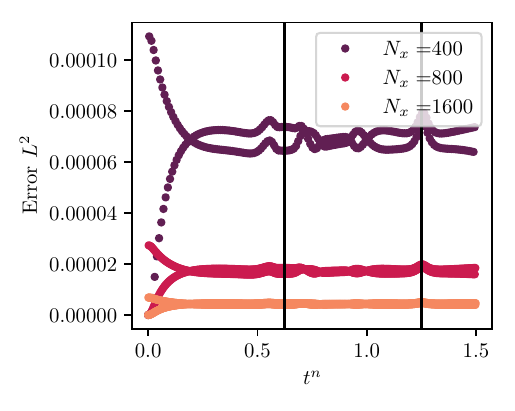}
        \end{center}\caption{\label{fig:12_error_time}$L^2$ error in time for $\testCaseId{2, 1}$ (left) and $\testCaseId{3, 1}$ (right) using $\numberSpacePoint$ grid points.}
    \end{figure}

    \item Initialization $\testCaseId{2, 1}$.
    The modal decomposition for $|\frequency\spaceStep| \ll 1$ is:
    \begin{multline*}
        \amplificationFactorFourier{\indicesTime}(\frequency\spaceStep) = (1 + \tfrac{1}{4}(\courantNumber^2 - 1)(\frequency\spaceStep)^2 + \bigO{|\frequency\spaceStep|^3}) e^{-\imaginaryUnit \advectionVelocity\gridPointTime{\indicesTime}\frequency(1+\bigO{|\frequency\spaceStep|^4})} \\
        + (-i\sqrt{3}(\courantNumber^2 - 1)(19\courantNumber + \sqrt{3}\sqrt{8-5\courantNumber^2})(\frequency\spaceStep)^2 + \bigO{|\frequency\spaceStep|^3}) (-1)^{\indicesTime}  e^{\tfrac{\imaginaryUnit\sqrt{3} \gridPointTime{\indicesTime}}{6} (\sqrt{3}\advectionVelocity + \sqrt{8\latticeVelocity^2 - 5\advectionVelocity^2})\frequency (1+\bigO{|\frequency\spaceStep|^2})} \\
        + (-i\sqrt{3}(\courantNumber^2 - 1)(-19\courantNumber + \sqrt{3}\sqrt{8-5\courantNumber^2})(\frequency\spaceStep)^2 + \bigO{|\frequency\spaceStep|^3}) (-1)^{\indicesTime}  e^{\tfrac{\imaginaryUnit\sqrt{3} \gridPointTime{\indicesTime}}{6} (\sqrt{3}\advectionVelocity - \sqrt{8\latticeVelocity^2 - 5\advectionVelocity^2})\frequency (1+\bigO{|\frequency\spaceStep|^2})}.
    \end{multline*}
    The physical mode and the parasitic modes are all present with a distortion of order two. 
    Now, the coefficients of the $\bigO{|\frequency\spaceStep|^2}$ perturbation of the parasitic modes are not one the opposite of the other, hence do not totally cancel out when the modes meet again when periodic boundary conditions are imposed, see \Cref{fig:12_error_time} for $\gridPointTime{\indicesTime} \approx 1.25$.
    We also have 
    \begin{equation*}
        |e^{-i\indicesTime\courantNumber\frequency\spaceStep} - \amplificationFactorFourier{\indicesTime}(\frequency\spaceStep)|^2 = 
        \begin{cases}
            \tfrac{\indicesTime^2}{36} \courantNumber^2 (\courantNumber^2 - 1)^2  (\frequency\spaceStep)^6 + \bigO{|\frequency\spaceStep|^8}, \qquad &\indicesTime ~ \textnormal{even}, \\
            \tfrac{1}{4}(\courantNumber^2 - 1)^2(\frequency\spaceStep)^4 + \alpha_{\courantNumber}(\indicesTime) (\frequency\spaceStep)^6 + \bigO{|\frequency\spaceStep|^8}, \qquad &\indicesTime ~ \textnormal{odd}, 
        \end{cases}
    \end{equation*}
    where $\alpha_{\courantNumber}(\indicesTime) = \bigO{\indicesTime^2}$ or more explicitly $\alpha_{\courantNumber}(\indicesTime) =  (\tfrac{\indicesTime^2}{9} - \tfrac{17\indicesTime}{72} + \tfrac{1}{9}  )\courantNumber^6 +  ( -\tfrac{11\indicesTime^2}{36} + \tfrac{23\indicesTime}{36} - \tfrac{25}{72}  )\courantNumber^4 +  (\tfrac{5\indicesTime^2}{18} - \tfrac{41\indicesTime}{72} + \tfrac{13}{36}  )\courantNumber^2 +  ( -\tfrac{\indicesTime^2}{12} + \tfrac{\indicesTime}{6} - \tfrac{1}{8}  )$.
    Only the odd steps carry the fixed term being the remaining trace of the Lax-Friedrichs scheme for $\indicesTime = 1$. Moreover, we see that the behavior of the error is radically different, as visible in \Cref{fig:12_error_time}, between even and odd steps. The former typically carry smaller errors. The initial decreasing behavior of the error for odd steps can be understood by noticing that the function $\alpha_{1/4}(\indicesTime)$ decreases in $\indicesTime$.
    \item Initialization $\testCaseId{3, 1}$. The conclusions are the same as $\testCaseId{2, 1}$. The expression for $|e^{-i\indicesTime\courantNumber\frequency\spaceStep} - \amplificationFactorFourier{\indicesTime}(\frequency\spaceStep)|^2$ in the even cases is even more different from the one in the odd cases compared to $\testCaseId{1, 2}$.
    \item Initialization $\testCaseId{3, 4}$.
    \begin{figure}[h]
        \begin{center}
            \includegraphics[scale = 0.99]{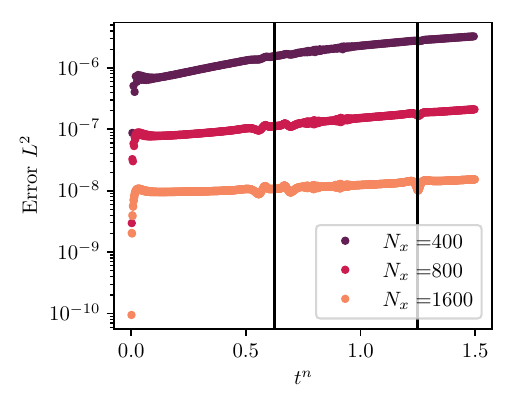}
            \includegraphics[scale = 0.99]{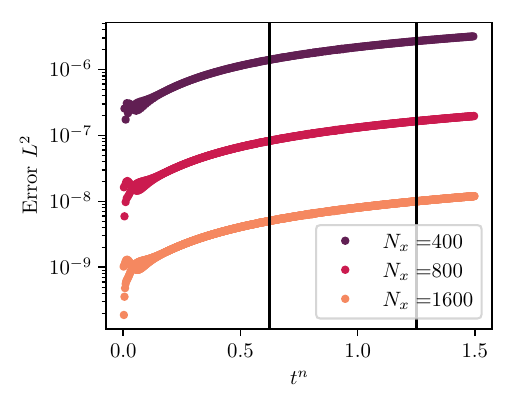}
        \end{center}\caption{\label{fig:43_error_time}$L^2$ error in time for $\testCaseId{3, 4}$ (left) and $\testCaseId{4, 3}$ (right) using $\numberSpacePoint$ grid points.}
    \end{figure}
    The modal decomposition is, for $|\frequency\spaceStep| \ll 1$:
    \begin{multline*}
        \amplificationFactorFourier{\indicesTime}(\frequency\spaceStep) = (1 -\tfrac{\courantNumber (\courantNumber^2-1)(\courantNumber-2)}{48} (\frequency\spaceStep)^4 + \bigO{|\frequency\spaceStep|^5}) e^{-\imaginaryUnit \advectionVelocity\gridPointTime{\indicesTime}\frequency(1+\bigO{|\frequency\spaceStep|^4})} \\
        + (\tfrac{\imaginaryUnit \courantNumber (\courantNumber^2-1)(\courantNumber-2)}{24\sqrt{8-5\courantNumber^2}} (\frequency\spaceStep)^3 + \bigO{|\frequency\spaceStep|^4}) (-1)^{\indicesTime}  e^{\tfrac{\imaginaryUnit\sqrt{3} \gridPointTime{\indicesTime}}{6} (\sqrt{3}\advectionVelocity + \sqrt{8\latticeVelocity^2 - 5\advectionVelocity^2})\frequency (1+\bigO{|\frequency\spaceStep|^2})} \\
        + (-\tfrac{\imaginaryUnit \courantNumber (\courantNumber^2-1)(\courantNumber-2)}{24\sqrt{8-5\courantNumber^2}} (\frequency\spaceStep)^3 + \bigO{|\frequency\spaceStep|^4})  (-1)^{\indicesTime}  e^{\tfrac{\imaginaryUnit\sqrt{3} \gridPointTime{\indicesTime}}{6} (\sqrt{3}\advectionVelocity - \sqrt{8\latticeVelocity^2 - 5\advectionVelocity^2})\frequency (1+\bigO{|\frequency\spaceStep|^2})}.
    \end{multline*}
    The physical mode is present with a distortion of order four: this is not what limits the order.
    What lowers the overall order to three is are the rapidly oscillating spurious modes which have amplitude $\bigO{|\frequency\spaceStep|^3}$. 
    However, for periodic boundary conditions, we see that the scheme converges at order four for a final time $\finalTime \approx 1.25$ (small basins on \Cref{fig:43_error_time}), thanks to the cancellation of spurious modes.
    Also
    \begin{equation*}
        |e^{-i\indicesTime\courantNumber\frequency\spaceStep} - \amplificationFactorFourier{\indicesTime}(\frequency\spaceStep)|^2 = 
        \begin{cases}
            \bigl ( \tfrac{\courantNumber(\courantNumber^2-1)(\courantNumber-2)}{24} \indicesTime\bigr )^2(\frequency\spaceStep)^8 + \bigO{|\frequency\spaceStep|^{10}}, \qquad &\indicesTime~\text{even}, \\
            \bigl ( \tfrac{\courantNumber(\courantNumber^2-1)(\courantNumber-2)}{24} (\indicesTime-1)\bigr )^2(\frequency\spaceStep)^8 + \bigO{|\frequency\spaceStep|^{10}}, \qquad &\indicesTime~\text{odd},
        \end{cases}
    \end{equation*}
    shows that---\confer{} \Cref{fig:43_error_time}---even and odd steps behave essentially in the same way.
    \item Initialization $\testCaseId{4, 3}$.
    The modal decomposition is
    \begin{multline*}
        \amplificationFactorFourier{\indicesTime}(\frequency\spaceStep) = (1 -\tfrac{\courantNumber (\courantNumber^2-1)(\courantNumber-2)}{48} (\frequency\spaceStep)^4 + \bigO{|\frequency\spaceStep|^5}) e^{-\imaginaryUnit \advectionVelocity\gridPointTime{\indicesTime}\frequency(1+\bigO{|\frequency\spaceStep|^4})} \\
        + \bigO{|\frequency\spaceStep|^4} (-1)^{\indicesTime}  e^{\tfrac{\imaginaryUnit\sqrt{3} \gridPointTime{\indicesTime}}{6} (\sqrt{3}\advectionVelocity + \sqrt{8\latticeVelocity^2 - 5\advectionVelocity^2})\frequency (1+\bigO{|\frequency\spaceStep|^2})} \\
        + \bigO{|\frequency\spaceStep|^4}  (-1)^{\indicesTime}  e^{\tfrac{\imaginaryUnit\sqrt{3} \gridPointTime{\indicesTime}}{6} (\sqrt{3}\advectionVelocity - \sqrt{8\latticeVelocity^2 - 5\advectionVelocity^2})\frequency (1+\bigO{|\frequency\spaceStep|^2})},
    \end{multline*}
    where the coefficients of the spurious waves (not given for the sake of compactness) are not one the opposite of the other even at leading order.
    Also
    \begin{equation*}
        |e^{-i\indicesTime\courantNumber\frequency\spaceStep} - \amplificationFactorFourier{\indicesTime}(\frequency\spaceStep)|^2 = 
        \begin{cases}
            \alpha_{\courantNumber}(\indicesTime)  (\frequency\spaceStep)^{10} + \bigO{|\frequency\spaceStep|^{12}}, \qquad &\indicesTime ~ \textnormal{even}, \\
            \bigl ( \tfrac{\courantNumber(\courantNumber^2-1)(\courantNumber-2)}{24}\bigr )^2 (\frequency\spaceStep)^8 +\beta_{\courantNumber}(\indicesTime) (\frequency\spaceStep)^{10}+ \bigO{|\frequency\spaceStep|^{12}}, \qquad &\indicesTime ~ \textnormal{odd}, 
        \end{cases}
    \end{equation*}
    where $\alpha_{\courantNumber}(\indicesTime) = \bigO{\indicesTime^2}$ and $\beta_{\courantNumber}(\indicesTime) = \bigO{\indicesTime^2}$. 
\end{itemize}

\subsection{Convergence with non-smooth initial data}\label{sec:testNonSmoothData}

\begin{figure}[h]
    \begin{center}
        \includegraphics{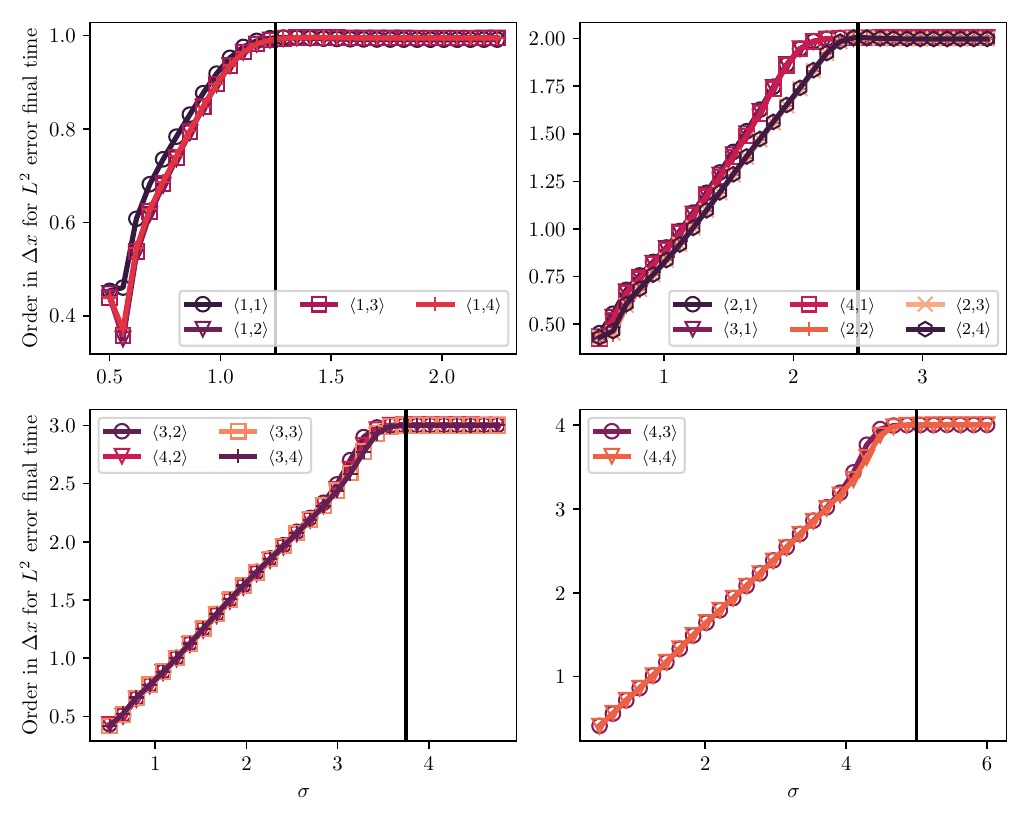}
    \end{center}\caption{\label{fig:ordersLowRegularity}Convergence rates in $\spaceStep$ for the $L^2$ error at final time $\finalTime = 0.2$, as function of $\sobolevRegularityInitialDatum$ when $\initialDatum \in \sobolevSpaceEllTwo{\sobolevRegularityInitialDatum-}(\reals)$. The black vertical line corresponds to $\sigma = \tfrac{5}{4}\min(\maximumOrder, \maximumOrder_2, \maximumOrder_1 + 1)$.}
\end{figure}

In the past, orders of convergence for non-smooth initial data were investigated in \cite{brenner1970stability, brenner1975besov, strikwerda2004finite, courtes17} for genuinely stable schemes.
As previously pointed out, the assumption that $\initialDatum \in \sobolevSpaceEllTwo{\maximumOrder+1}(\reals)$ is often too strong---especially when the overall order $\min(\maximumOrder, \maximumOrder_2, \maximumOrder_1 + 1)$ is way smaller than $\maximumOrder$---to actually observe the expected order.
Indeed, this order can be observed for non-smooth initial data being less that $\sobolevSpaceEllTwo{5}$.
We now perform a numerical simulation to verify this fact. 
The conditions are, as usual, domain  $[-1, 1]$ endowed with periodic boundary conditions, $\courantNumber = 1/4$, and $\finalTime = 0.2$.
We consider the initial functions $\initialDatum(\spaceVariable) = \cos(\pi\spaceVariable)^{\sobolevRegularityInitialDatum - 1/2}\indicatorFunction{(-1, 1)}(2\spaceVariable)$ for $\sobolevRegularityInitialDatum \geq 1$, inspired by \cite[Example II, Section 2.4]{brenner1975besov}.
These are such that $\initialDatum \in \besovSpace{\sobolevRegularityInitialDatum}{2}{\infty} (\reals)$ in terms of Besov spaces or $\initialDatum \in \sobolevSpaceEllTwo{\sobolevRegularityInitialDatum-} (\reals)$ in terms of Sobolev spaces.
We measure the convergence rate in $\spaceStep$ with respect to the $L^2$ norm of the error at final time, obtaining the results in \Cref{fig:ordersLowRegularity}.
When the overall scheme is less than fourth-order accurate (maximal order), the order saturates before the smoothness $\initialDatum \in \sobolevSpaceEllTwo{5}(\reals)$ for the initial datum is reached. 
The value of $\sobolevRegularityInitialDatum$ from which the order saturates can be estimated using \cite[Corollary 10.3.2]{strikwerda2004finite} (notice that the error estimates in this result contain rarely observed logarithmic terms in $\spaceStep$) by $\sobolevRegularityInitialDatum \approx \tfrac{5}{4}\min(\maximumOrder, \maximumOrder_2, \maximumOrder_1 + 1)$.
This is obtained by intersecting the line $\sobolevRegularityInitialDatum \mapsto \tfrac{4}{5} \sobolevRegularityInitialDatum$, associated with a fourth-order (bulk) scheme, with the fact that the order is limited by $\min(\maximumOrder, \maximumOrder_2, \maximumOrder_1 + 1)$. 
This is the intuitive best approximation of the minimal smoothness that we were able to find.
Slight deviations from this behavior can be understood by looking at \eqref{eq:greenDecompositionConsistency} or \eqref{eq:totalErrorInequality}.
For the selected final time, there might still be an important (yet difficult to describe) effect of the part of the solution associated with $\greenFunction{\indicesTime}{2}$ and $\greenFunction{\indicesTime}{1}$. If we take larger $\finalTime$ (simulations not presented in the paper), the term brought by $\rootAmplificationPolynomialFourier{1}^{\indicesTime}$ becomes dominant and the observed behavior adhere better and better to the estimate according to which the order is $\min(\min(\maximumOrder, \maximumOrder_2, \maximumOrder_1 + 1), \tfrac{4}{5}\sobolevRegularityInitialDatum)$ for $\initialDatum \in \sobolevSpaceEllTwo{\sobolevRegularityInitialDatum-}(\reals)$.

\section{Role of the round-off errors}\label{sec:roundoffErrors}

The discussion that has been developed hitherto has been conducted as if actual numerical simulations were done at arbitrary machine precision, without being affected by round-off errors.
However, the instability cancellation elucidated in \Cref{sec:stability} could not actually take place in floating-point arithmetic, due to round-off errors.
As pointed out by \cite{trefethen1992definition}, this does however not question the interest of the previous discussions.
Quite the opposite, this gives one more reason why, besides the order of consistency/convergence, one generally avoids weakly unstable schemes: they could lead unpredictable behaviors in presence of floating point numbers.

\begin{figure}[h]
    \hspace{1.5cm}LBM \eqref{eq:collisionLatticeBoltzmann}/\eqref{eq:streamLatticeBoltzmann} with \eqref{eq:FourthOrderInitializationLBM}, $\delta = 0$ \hspace{1.8cm}FD \eqref{eq:FourthOrderLeapFrog} with \eqref{eq:FourthOrderInitializationLBM}, $\delta = 0$ \hspace{3.35cm}OS4
    \begin{center}
        \includegraphics{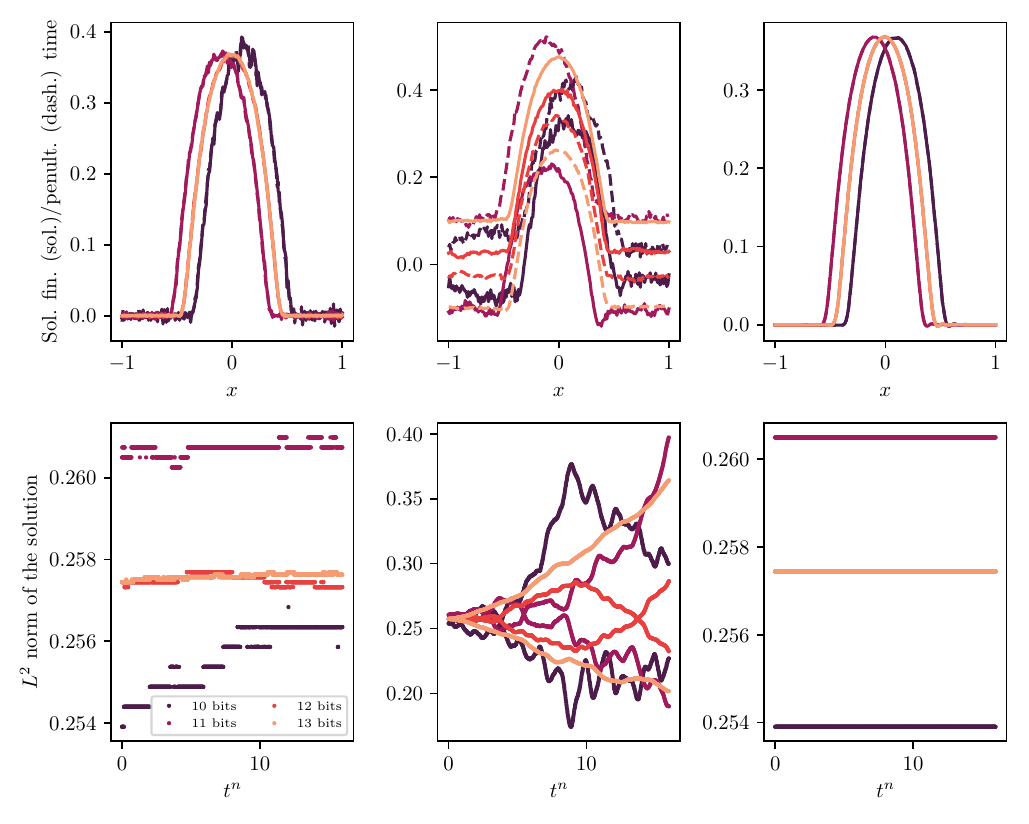}
    \end{center}\caption{\label{fig:leapFrogPrecision}Test varying the floating point precision.}
\end{figure}

We perform the same test as we did in \Cref{sec:NumericalExperiment} except for the fact that we use a longer final time of $\finalTime = 16$ and fix the number of points in the domain to $\numberSpacePoint = 200$.
Computations are carried out using unusually low floating-point precisions, ranging from 10 to 13 bits, to be compared with the usual 64 bits double precision in \texttt{Python}.
We employ the stable lattice Boltzmann scheme \eqref{eq:collisionLatticeBoltzmann}/\eqref{eq:streamLatticeBoltzmann} with \eqref{eq:FourthOrderInitializationLBM} using $\delta = 0$, its weakly unstable corresponding Finite Difference scheme \eqref{eq:FourthOrderLeapFrog} with initialization \eqref{eq:FourthOrderInitializationLBM} using $\delta  = 0$, and the one-step stable OS4 scheme.
Notice that the first two procedures are perfectly equivalent ``on paper''.
In \Cref{fig:leapFrogPrecision}, we show the solution both at the final and penultimate time step, and its $L^2$ norm as function of time, which we take as a measure of the instability when applicable.

We observe that for the weakly unstable scheme \eqref{eq:FourthOrderLeapFrog} with initialization \eqref{eq:FourthOrderInitializationLBM} with $\delta = 0$, especially for very low floating point accuracy, the solution is totally shifted either upward or downward, according to the parity of the time step $\indicesTime$. This is the reason why we decided to show both the solution at the final and penultimate time step.
The fact that the solution is approximately shifted by a constant proportional to $(-1)^{\indicesTime} \indicesTime$ stems from the following features of the scheme.
The constant nature of the shift is dictated by the unstable mode corresponding to the frequency zero, and thus corresponding to an unstable mode in the physical space under the form of a constant function: by inverse Fourier transform $(2\pi)^{-1/2} \int_{|\frequency| \leq \pi} e^{i\indicesSpace \frequency}\delta_{0}(\frequency)\differential{\frequency} = (2\pi)^{-1/2} $ for every $\indicesSpace \in \relatives$.
The oscillation trending like $(-1)^{\indicesTime}$ comes from the fact that the unstable modes are associated with the spurious roots, which are such that $\rootAmplificationPolynomialFourier{2}(0)^{\indicesTime} = \rootAmplificationPolynomialFourier{3}(0)^{\indicesTime} = (-1)^{\indicesTime}$. The growth behaving roughly proportionally to $\indicesTime$ comes from the instability due to the multiple nature of the roots at frequency zero.
For the stable OS4 scheme, the simulation remains perfectly controlled and the norm of the solution does not increase even using a very rough floating point arithmetic.
The solution for the lattice Boltzmann scheme \eqref{eq:collisionLatticeBoltzmann}/\eqref{eq:streamLatticeBoltzmann} with \eqref{eq:FourthOrderInitializationLBM} and $\delta = 0$ is completely different from the one of its corresponding Finite Difference scheme and remains totally stable, as it should be.
This indicates that the former implementation has a better backward stability compared to the latter.
Small oscillations, compared to the OS4 scheme, come from the lack of dissipation of the lattice Boltzmann scheme on the whole spectrum of frequencies.

\begin{remark}[On the enhanced stability of the lattice Boltzmann algorithm]
    To understand the enhanced stability of \eqref{eq:collisionLatticeBoltzmann}/\eqref{eq:streamLatticeBoltzmann}/\eqref{eq:FourthOrderInitializationLBM} compared to \eqref{eq:FourthOrderLeapFrog}/\eqref{eq:FourthOrderInitializationLBM}, we come back to the way of turning lattice Boltzmann schemes into Finite Difference ones.
    In the entire paper, the problematic frequency has been $\frequency = 0$, where the matrix giving the lattice Boltzmann scheme becomes 
    \begin{equation*}
        \fourierTransformed{\schemeMatrix}(0) = 
        \begin{bmatrix}
            1 & 0 & 0\\
            2\courantNumber & - 1 & 0 \\
            2(\courantNumber^2 - 1) & 0 & -1
        \end{bmatrix}, \qquad \textnormal{hence} \qquad 
        \fourierTransformed{\schemeMatrix}(0)^{\indicesTime} = 
        \begin{bmatrix}
            1 & 0 & 0\\
            \courantNumber (1-(-1)^{\indicesTime}) & (-1)^{\indicesTime} & 0 \\
            (\courantNumber^2 - 1) (1-(-1)^{\indicesTime}) & 0 & (-1)^{\indicesTime}
        \end{bmatrix},
    \end{equation*}
    thus is uniformly power bounded and thus the original lattice Boltzmann scheme genuinely stable, since 
    \begin{equation}\label{eq:tmp1}
        \fourierTransformed{\discrete{\firstMomentLetter}}^{\indicesTime}(0) = \transpose{\canonicalBasisVector{1}}\fourierTransformed{\schemeMatrix}(0)^{\indicesTime} \transpose{(\fourierTransformed{\discrete{\firstMomentLetter}}^{0}(0), \fourierTransformed{\discrete{\secondMomentLetter}}^{0}(0), \fourierTransformed{\discrete{\thirdMomentLetter}}^{0}(0))} = \fourierTransformed{\discrete{\firstMomentLetter}}^{0}(0).
    \end{equation}
    Observe that $\fourierTransformed{\schemeMatrix}(0)^{2} = \identityMatrix$, thus the polynomial $\timeShiftOperator^2 - 1$ dividing $\determinant{\timeShiftOperator\identityMatrix - \fourierTransformed{\schemeMatrix}(0)}$ annihilates $\fourierTransformed{\schemeMatrix}(0)$: it is its minimal polynomial. 
    This difference between characteristic and minimal polynomials is coherent with the fact that $\fourierTransformed{\schemeMatrix}(0)$ is not similar to its companion matrix $\companionMatrix(0)$, with the former being diagonalisable and the latter no.
    As pointed out in \cite{bellotti2022finite}, the fact that a polynomial annihilates the scheme matrix allows to recover the corresponding Finite Difference scheme through it.
    However, this should be true for all frequency $\frequency$, which is not the case here for $\timeShiftOperator^2 - 1$.
    Nevertheless, for the sole unstable frequency $\frequency = 0$ (or equivalently, constant solutions), the dynamics of the lattice Boltzmann scheme \eqref{eq:collisionLatticeBoltzmann}/\eqref{eq:streamLatticeBoltzmann}, as far as $\firstMomentDiscrete{}{}$ is concerned, can be rewritten using something simpler (and more stable) than \eqref{eq:FourthOrderLeapFrog}, which reads $ \solutionDiscreteFourier{\indicesTime + 1}(0) = \solutionDiscreteFourier{\indicesTime - 1}(0)$.
    This scheme is genuinely stable, since it has two eigenvalues $\pm 1$ on the unit circle being distinct.
    It has to be interpreted as a constraint fulfilled by the solution $\firstMomentDiscrete{}{}$ of \eqref{eq:collisionLatticeBoltzmann}/\eqref{eq:streamLatticeBoltzmann}. The constraint given by \eqref{eq:FourthOrderLeapFrog} is also satisfied, but it yields an overconstrained mechanism, more prone to instabilities when the exact compensation resulting in overall stability could not take place, due to floating-point numbers.
    Pushing this way of rasoning even further, we have observed in \cite{bellotti2022finite} that if a polynomial annihilates (again for all $\frequency$) the first row of the scheme matrix, hence it divides its minimal polynomial, it yields a corresponding Finite Difference scheme as well. In our case, for $\fourierTransformed{\schemeMatrix}(0)$, the polynomial is simply $\timeShiftOperator - 1$, resulting in $ \solutionDiscreteFourier{\indicesTime + 1}(0) = \solutionDiscreteFourier{\indicesTime}(0)$, which is perfectly compatible with \eqref{eq:tmp1} and extremely stable.
\end{remark}

\FloatBarrier

\section{Application: ``nested'' kinetic schemes for non-linear systems}\label{sec:applicationKinetic}

We now showcase how a fourth-order scheme for the linear transport equation can be used, utilizing a Jin-Xin relaxation system \cite{jin1995relaxation}, as the the basic brick to approximate the solution of a non-linear system of conservation laws \cite{aregba2000discrete}. 
This application is strongly inspired by the work of \cite{coulette2019high}.
We call the schemes ``nested'' because we use an external kinetic scheme based on the Jin-Xin relaxation system to deal with the non-linearity and then inner lattice Boltzmann schemes---a special type of kinetic schemes---to solve the transport equations appearing in the Jin-Xin system.

\subsection{Jin-Xin relaxation system}

We tackle the solution of the system on $\vectorial{\conservedMomentsSystemConservationLaws} : \reals \times \reals \to \reals^{\numberConservationLaws}$ which reads
\begin{equation}\label{eq:systemConservationLaws}
    \partial_{\timeVariable} \vectorial{\conservedMomentsSystemConservationLaws} + \partial_{\spaceVariable} \vectorial{\fluxSystemConservationLaws}(\vectorial{\conservedMomentsSystemConservationLaws}) = 0,
\end{equation}
where $\vectorial{\fluxSystemConservationLaws}: \reals^{\numberConservationLaws} \to \reals^{\numberConservationLaws}$ is a smooth and possibly non-linear flux.
To give a simple example, taking $\numberConservationLaws = 1$ and $\fluxSystemConservationLaws(\conservedMomentsSystemConservationLaws) = \tfrac{1}{2}\conservedMomentsSystemConservationLaws^2$ gives the inviscid Burgers equation.
In order to isolate the non-linearity of the problem into a local relaxation term which is easily tractable, we consider the associated Jin-Xin relaxation system \cite{jin1995relaxation}
\begin{equation}\label{eq:relaxationSystem}
    \begin{cases}
        \partial_{\timeVariable} \vectorial{\conservedMomentsSystemConservationLaws} + \partial_{\spaceVariable} \vectorial{\nonConservedMomentsSystemConservationLaws} = 0, \\
        \partial_{\timeVariable} \vectorial{\nonConservedMomentsSystemConservationLaws} + \kineticVelocity^2 \partial_{\spaceVariable} \vectorial{\conservedMomentsSystemConservationLaws} = -\frac{1}{\relaxationTime}(\vectorial{\nonConservedMomentsSystemConservationLaws} - \vectorial{\fluxSystemConservationLaws}(\vectorial{\conservedMomentsSystemConservationLaws})),
    \end{cases}
\end{equation}
with $\kineticVelocity > 0$ a kinetic velocity and $\relaxationTime \geq 0$ a relaxation time, whose formal limit for $\relaxationTime \to 0$ gives back \eqref{eq:systemConservationLaws}.
Using the change of basis $\vectorial{\conservedMomentsSystemConservationLaws} = \vectorial{\distributionFunctionLetter}^+ + \vectorial{\distributionFunctionLetter}^-$ and $\vectorial{\nonConservedMomentsSystemConservationLaws} = \kineticVelocity (\vectorial{\distributionFunctionLetter}^+ - \vectorial{\distributionFunctionLetter}^-)$, the relaxation system \eqref{eq:relaxationSystem} can be recast into its kinetic form
\begin{equation}\label{eq:relaxationSystemKinetic}
    \partial_{\timeVariable} \vectorial{\distributionFunctionLetter}^{\pm} \pm \kineticVelocity\partial_{\spaceVariable} \vectorial{\distributionFunctionLetter}^{\pm} = -\frac{1}{\relaxationTime} (\vectorial{\distributionFunctionLetter}^{\pm} - \vectorial{\distributionFunctionLetter}^{\pm, \atEquilibrium}(\vectorial{\conservedMomentsSystemConservationLaws} )), \qquad \textnormal{where} \quad \vectorial{\distributionFunctionLetter}^{\pm, \atEquilibrium}(\vectorial{\conservedMomentsSystemConservationLaws} ) = \frac{1}{2} \vectorial{\conservedMomentsSystemConservationLaws} \pm \frac{1}{2\kineticVelocity} \vectorial{\fluxSystemConservationLaws}(\vectorial{\conservedMomentsSystemConservationLaws}).
\end{equation}

\subsection{Splitting and numerical schemes}

Notice that the left hand side of \eqref{eq:relaxationSystemKinetic} is nothing but a linear transport equation at velocities $\pm \kineticVelocity$, which can therefore be approximated using the scheme \eqref{eq:FourthOrderLeapFrog} introduced in the first part of the paper or, for efficient computations taking advantage of the peculiar structure of lattice Boltzmann methods, using \eqref{eq:collisionLatticeBoltzmann} and \eqref{eq:streamLatticeBoltzmann}.
We then split \eqref{eq:relaxationSystemKinetic} into its transport (T) part $\partial_{\timeVariable} \vectorial{\distributionFunctionLetter}^{\pm} \pm \kineticVelocity\partial_{\spaceVariable} \vectorial{\distributionFunctionLetter}^{\pm} = 0$, and relaxation part (R) $\partial_{\timeVariable} \vectorial{\distributionFunctionLetter}^{\pm} = -\tfrac{1}{\relaxationTime} (\vectorial{\distributionFunctionLetter}^{\pm} - \vectorial{\distributionFunctionLetter}^{\pm, \atEquilibrium}(\vectorial{\conservedMomentsSystemConservationLaws} ))$.
The transport part (T) shall be solved using $\numberSubstepTransportSolver \geq 3$ steps of the fourth-order lattice Boltzmann scheme for the transport equation or its corresponding Finite Difference scheme.
The reason to perform more than three steps to reach the desired final time $\splittingTimeStep > 0$ is dictated by the fact that we want the multi-step method to produce fourth-order accurate approximations. 
Therefore, the time-step $\timeStep$ for this sub-routine will be given by $\timeStep = \splittingTimeStep/\numberSubstepTransportSolver$.
The associated discrete operator is denoted $\operatorTransport(\splittingTimeStep)$.
The relaxation part (R) is solved using the following trapezoidal quadrature, see \cite{dellar2013interpretation}:
\begin{multline*}
    \int_{0}^{\splittingTimeStep} \partial_{\timeVariable} \vectorial{\distributionFunctionLetter}^{\pm} \differential{\timeVariable} = \vectorial{\distributionFunctionLetter}^{\pm}(\splittingTimeStep) - \vectorial{\distributionFunctionLetter}^{\pm}(0) = -\frac{1}{\relaxationTime} \int_{0}^{\splittingTimeStep} \partial_{\timeVariable}  (\vectorial{\distributionFunctionLetter}^{\pm} - \vectorial{\distributionFunctionLetter}^{\pm, \atEquilibrium}(\vectorial{\conservedMomentsSystemConservationLaws} ))\differential{\timeVariable} \\
    = -\frac{\splittingTimeStep}{2\relaxationTime} ((\vectorial{\distributionFunctionLetter}^{\pm} - \vectorial{\distributionFunctionLetter}^{\pm, \atEquilibrium}(\vectorial{\conservedMomentsSystemConservationLaws} ))(\splittingTimeStep) + (\vectorial{\distributionFunctionLetter}^{\pm} - \vectorial{\distributionFunctionLetter}^{\pm, \atEquilibrium}(\vectorial{\conservedMomentsSystemConservationLaws} ))(0)) + \bigO{\splittingTimeStep^2}.
\end{multline*}
Using the fact that the relaxation phase conserves $\vectorial{\conservedMomentsSystemConservationLaws}$ and thus $\vectorial{\distributionFunctionLetter}^{\pm, \atEquilibrium}(\vectorial{\conservedMomentsSystemConservationLaws}(\splittingTimeStep) ) = \vectorial{\distributionFunctionLetter}^{\pm, \atEquilibrium}(\vectorial{\conservedMomentsSystemConservationLaws} (0))$, the algorithm can be kept explicit and thus reads
\begin{equation*}
    \vectorial{\distributionFunctionLetter}^{\pm}(\splittingTimeStep) = \frac{2\relaxationTime - \splittingTimeStep}{2\relaxationTime + \splittingTimeStep} \vectorial{\distributionFunctionLetter}^{\pm}(0) + \frac{2\splittingTimeStep}{2\relaxationTime + \splittingTimeStep} \vectorial{\distributionFunctionLetter}^{\pm, \atEquilibrium}(\vectorial{\conservedMomentsSystemConservationLaws} (0)) \xrightarrow[\relaxationTime \to 0]{} - \vectorial{\distributionFunctionLetter}^{\pm}(0) + 2 \vectorial{\distributionFunctionLetter}^{\pm, \atEquilibrium}(\vectorial{\conservedMomentsSystemConservationLaws} (0)).
\end{equation*}
This relaxation limit $\relaxationTime \to 0$ is unsurprisingly similar to \eqref{eq:collisionLatticeBoltzmann} for the non-conserved moments when the relaxation parameters $\relaxationParameterSecondMoment$ and $\relaxationParameterThirdMoment$ equal two.
Notice that the operator associated with this relaxation step in the relaxation limit $\relaxationTime \to 0$ does not depend on $\splittingTimeStep$.
It shall be denoted by $\operatorRelaxation$ and is an involution, meaning that $\operatorRelaxation^2 = \vectorial{\discrete{I}}$.
Following \cite{coulette2019high}, we construct the basic brick of a splitting procedure, called $\basicSplittingOperator$, given by
\begin{equation}\label{eq:basicBrickSplitting}
    \basicSplittingOperator(\splittingTimeStep) = \operatorTransport \Bigl ( \frac{\splittingTimeStep}{4}\Bigr ) \operatorRelaxation \operatorTransport \Bigl ( \frac{\splittingTimeStep}{2}\Bigr ) \operatorRelaxation \operatorTransport \Bigl ( \frac{\splittingTimeStep}{4}\Bigr ).
\end{equation}
This operator is time-symmetric up to order four: the transformation $\splittingTimeStep \mapsto -\splittingTimeStep$ makes it its own inverse operator.
The final operator to be applied is constructed using a fourth-order five-stages symmetric Suzuki splitting \cite{mclachlan2002splitting} given by 
\begin{equation*}
    \operatorSuzuki (\splittingTimeStep) = \basicSplittingOperator\Bigl (\frac{1}{4-4^{1/3}}\splittingTimeStep \Bigr )^2 \basicSplittingOperator\Bigl (-\frac{4^{1/3}}{4-4^{1/3}}\splittingTimeStep \Bigr )\basicSplittingOperator\Bigl (\frac{1}{4-4^{1/3}}\splittingTimeStep \Bigr )^2.
\end{equation*}
Notice that the central stage features a negative time-step.
We handle this point using the fact that the relaxation $\operatorRelaxation$ appearing in \eqref{eq:basicBrickSplitting} does not depend on the time-step and that the linear transport equation solved by $\operatorTransport$ is reversible in time, hence we just switch $\pm \kineticVelocity \mapsto \mp \kineticVelocity$.

\subsection{Numerical experiments}

\begin{figure}[h]
    \begin{center}
        \includegraphics[scale = 0.99]{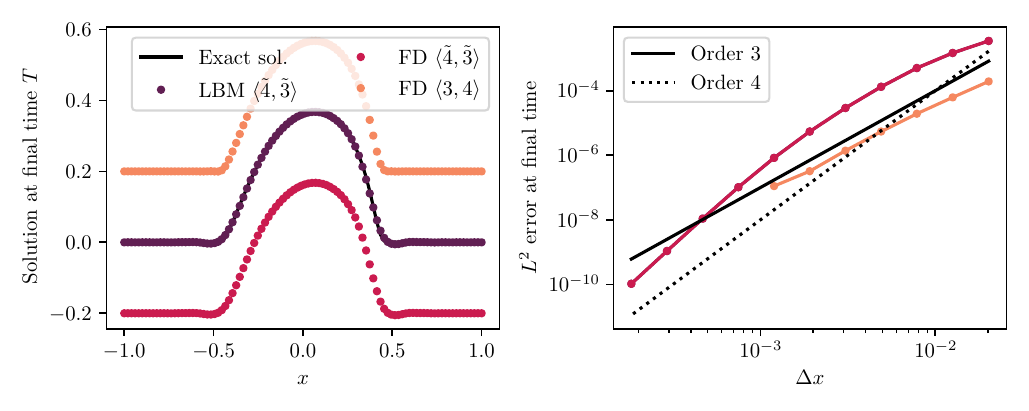}
    \end{center}\caption{\label{fig:kinetic_solver_Burgers}Left: solution of the Burgers equation at final time $\finalTime = 0.2$ using $\numberSpacePoint = 100$ grid points for different choices of initialization using the lattice Boltzmann algorithm and the corresponding Finite Difference scheme. The notation $\testCaseId{\tilde{4}, \tilde{3}}$ indicates \eqref{eq:FourthOrderInitializationLBM} with $\delta = 0$. Curves are shifted in order to distinguish between them. Right: convergence of the $L^2$ error at final time for the Burgers equation.}
\end{figure}

To test the previously described numerical procedure and---in particular---the fact that it ensures an overall fourth-order method for systems of non-linear PDEs, we consider the inviscid Burgers equation using the same setting of \Cref{sec:NumericalExperiment} as far computational domain, initial datum $\initialDatum$, and final time $\finalTime$ are concerned.
We perform numerical simulations considering $\splittingTimeStep = \spaceStep$, $\numberSubstepTransportSolver = 6$, and $\advectionVelocity = 1$.
This choice respects the CFL condition, \confer{} \Cref{prop:StabilityThirdLeapFrog}, for the splitting stage with the largest step, given by 
\begin{equation*}
    \frac{|\advectionVelocity|4^{1/3}}{2\numberSubstepTransportSolver (4 - 4^{1/3})} < \frac{1}{2}.
\end{equation*}

The results are presented in \Cref{fig:kinetic_solver_Burgers}.
We simulate using the original lattice Boltzmann scheme \eqref{eq:collisionLatticeBoltzmann}/\eqref{eq:streamLatticeBoltzmann} with initialization \eqref{eq:FourthOrderInitializationLBM} and $\delta = 0$, its corresponding Finite Difference scheme \eqref{eq:FourthOrderLeapFrog} with initialization \eqref{eq:FourthOrderInitializationLBM} and $\delta = 0$ (which are indeed equivalent), and \eqref{eq:FourthOrderLeapFrog} with initialization $\testCaseId{3, 4}$.
We observe order four when using \eqref{eq:FourthOrderInitializationLBM} with $\delta = 0$ and order three when using $\testCaseId{3, 4}$, as expected from the previously conducted analyses.
However, we see that $\testCaseId{3, 4}$ eventually leads to instabilities when $\spaceStep$ decreases.
This is probably due to the fact that \eqref{eq:FourthOrderLeapFrog} is non-dissipative for every mode (all its eigenvalue are on the unit circle for every frequency), both physical and parasitic, and that performing a large number of operations, due to the quite complex splitting procedure, triggers the instability of \eqref{eq:FourthOrderLeapFrog} due to round-off errors.
Moreover, we have previously observed that the choice $\testCaseId{3, 4}$ excites the unstable parasitic modes more (\idEst{} at order $\bigO{\spaceStep^3}$) than $\testCaseId{4, 3}$, $\testCaseId{4, 4}$ or \eqref{eq:FourthOrderInitializationLBM} (\idEst{} at order $\bigO{\spaceStep^4}$).

These numerical experiments make the interest of developing high-order lattice Boltzmann schemes for the linear transport equation clear, for this allows to tackle non-linear equations in an efficient manner, thanks to a Jin-Xin relaxation.
However, these experiments also show that genuine stability, and possibly some dissipation, are highly desirable property when working with floating point arithmetic.

\section*{Acknowledgements}

The author would like to deeply thank C. Court\`es (University of Strasbourg) for the discussion on this paper, L. Fran\c cois (ONERA) for pointing out that the features investigated in the paper come from the lack of genuine stability, V. Michel-Dansac (INRIA) for pointing out \Cref{rem:tryingToStabilize}, P. Helluy and L. Navoret (University of Strasbourg) for the help with Suzuki splittings, and D. Stantejsky (McMaster University) for his prompt advice on Sobolev spaces and Fourier analysis.
The author's post-doc position is funded by IRMIA++ from University of Strasbourg.

%

\bibliographystyle{alpha}
\bibliography{biblio}

\appendix

\section{Proof of \Cref{prop:StabilityThirdLeapFrog}}\label{app:ProofStabilityLeapFrog}

We use the iterative procedure by \cite{miller1971location} to determine if $\amplificationPolynomial{\frequency\spaceStep}{\timeShiftOperator}$ is a simple \emph{von Neumann} polynomial, namely it has roots inside the closed unit disk and those on the unit circle are simple.
    Let us set $\millerPolynomial{3}(\timeShiftOperator) = \amplificationPolynomial{\frequency\spaceStep}{\timeShiftOperator}$ by \eqref{eq:amplificationPolynomialFourthOrderLeapFrog}.
    We obtain that $\millerPolynomialConjugate{3}(\timeShiftOperator) = -\timeShiftOperator^3 - \commonOperator(\frequency\spaceStep)\timeShiftOperator^{2} + \conjugate{\commonOperator}(\frequency\spaceStep)\timeShiftOperator + 1$. This results in $\millerPolynomial{2}(\timeShiftOperator) = \timeShiftOperator^{-1}(\millerPolynomialConjugate{3}(0)\millerPolynomial{3}(\timeShiftOperator) - \millerPolynomial{3}(0)\millerPolynomialConjugate{3}(\timeShiftOperator)) \equiv 0$, hence we have to show that 
    \begin{equation}\label{eq:derivativePolynomial}
        \millerPolynomialOther{2}(\timeShiftOperator) \definitionEquality \textnormal{d}_{\timeShiftOperator} \millerPolynomial{3}(\timeShiftOperator) = 3\timeShiftOperator^2 + 2\commonOperator(\frequency\spaceStep)\timeShiftOperator - \conjugate{\commonOperator}(\frequency\spaceStep)
    \end{equation}
    is a \emph{Schur} polynomial, namely its roots belong to the open unit disk.
    We have $\millerPolynomialOtherConjugate{2}(\timeShiftOperator) = -\commonOperator(\frequency\spaceStep)\timeShiftOperator^2 + 2\conjugate{\commonOperator}(\frequency\spaceStep)\timeShiftOperator  +3$.
    A first condition to check is $|\millerPolynomialOther{2}(\timeShiftOperator) | < |\millerPolynomialOtherConjugate{2}(0)| $, giving $|\commonOperator(\frequency\spaceStep)| < 3$.
    It is simpler to work with the square of the modulus, which provides 
    \begin{equation}\label{eq:StabilityInequalityLeapFrog}
        (4\courantNumber^4 - 17\courantNumber^2 + 4)\cos^2(\frequency\spaceStep) + 2(-4\courantNumber^4 + 5\courantNumber^2 - 1)\cos(\frequency\spaceStep) + (4\courantNumber^4 + 7\courantNumber^2 - 20) < 0.
    \end{equation}
    Calling $\cosineShorthand \definitionEquality \cos(\frequency\spaceStep) \in [-1, 1]$, the previous equation gives a quadratic inequality on $\cosineShorthand$ to be satisfied on $[-1, 1]$. 
    Let us study it according to the sign of the leading term.
    \begin{itemize}
        \item $4\courantNumber^4 - 17\courantNumber^2 + 4 = (2\courantNumber-1)(2\courantNumber+1)(\courantNumber-2)(\courantNumber+2) > 0$, which is equivalent to $|\courantNumber| < 1/2$ or $|\courantNumber| > 2$. Under this condition, the maximum of the left-hand side in \eqref{eq:StabilityInequalityLeapFrog} is reached on the boundary of $[-1, 1]$.
        Taking $\cosineShorthand = -1$ provides $16\courantNumber^4 - 20\courantNumber^2 - 14 < 0$, which solution is $|\courantNumber| < \sqrt{7}/2 \approx 1.3229$. 
        Considering $\cosineShorthand = 1$ trivially gives $-18 < 0$.
        In this case, the overall condition is $|\courantNumber| < 1/2$.
        \item $4\courantNumber^4 - 17\courantNumber^2 + 4 = (2\courantNumber-1)(2\courantNumber+1)(\courantNumber-2)(\courantNumber+2) \leq 0$, which is equivalent to $1/2 \leq |\courantNumber | \leq 2$. In this case, the maximum of the left-hand side of \eqref{eq:StabilityInequalityLeapFrog} is reached inside $[-1, 1]$. The value of the inequality \eqref{eq:StabilityInequalityLeapFrog} on the maximum is given by $4(-4\courantNumber^4 + 5\courantNumber^2 - 1)^2 - 4(4\courantNumber^4 - 17\courantNumber^2 + 4)(4\courantNumber^4 + 7\courantNumber^2 - 20) < 0$.
        The solutions are $1/2 < |\courantNumber| < \sqrt{3/2} \approx 1.2247$.
        Overall, we the condition is $1/2 < |\courantNumber| < \sqrt{3/2}$.
    \end{itemize}
    We see that for $|\courantNumber| = 1/2$, the inequality \eqref{eq:StabilityInequalityLeapFrog} is also verified. 
    Therefore, from the previous discussion, the condition we obtain is $|\courantNumber| < \sqrt{3/2}$.
    The next step in the process is to check that $\millerPolynomialOther{1}(\timeShiftOperator) = (9 - |\commonOperator(\frequency\spaceStep)|^2) \timeShiftOperator + 6\commonOperator(\frequency\spaceStep) + 2 \conjugate{\commonOperator}(\frequency\spaceStep)^2$ is a \emph{Schur} polynomial as well.
    This condition reads $|6\commonOperator(\frequency\spaceStep) + 2 \conjugate{\commonOperator}(\frequency\spaceStep)^2|^2 -  (9 - |\commonOperator(\frequency\spaceStep)|^2)^2 < 0$.
    Observe that this does not hold for $\frequency = 0$ (or $\cosineShorthand = 1$), because here we have multiple roots on the unit circle.
    For $-1\leq \cosineShorthand <1$:
    \begin{multline*}
        (16\courantNumber^8 - 136\courantNumber^6 + 321 \courantNumber^4 - 136 \courantNumber^2 + 16)\cosineShorthand^4 - 4(16\courantNumber^8 - 64 \courantNumber^6 + 195 \courantNumber^4 - 127 \courantNumber^2 - 20) \cosineShorthand^3 \\
        + 3(32 \courantNumber^8 + 16\courantNumber^6 + 30\courantNumber^4 - 137\courantNumber^2 + 32) \cosineShorthand^2 - 4 (16\courantNumber^8 + 80\courantNumber^6 - 219\courantNumber^4 + 107\courantNumber^2 + 16)\cosineShorthand \\
        +16\courantNumber^8 + 152 \courantNumber^6 - 507\courantNumber^4 + 467 \courantNumber^2 - 128 < 0.
    \end{multline*}
    As previously discussed, $\cosineShorthand = 1$ is a zero of the left-hand side, thus we can factorize it out to yield 
    \begin{multline}\label{eq:StabilityThirdLeapFrog}
        (16 \courantNumber^8 - 136 \courantNumber^6 + 321 \courantNumber^4 - 136 \courantNumber^2 + 16) \cosineShorthand^3 - 3 (16 \courantNumber^8 - 40 \courantNumber^6 + 153 \courantNumber^4 - 124 \courantNumber^2 - 32) \cosineShorthand^2 \\
         + 3 (16 \courantNumber^8 + 56 \courantNumber^6 - 123 \courantNumber^4 - 13 \courantNumber^2 + 64) \cosineShorthand -16 \courantNumber^8 - 152 \courantNumber^6 + 507 \courantNumber^4 - 467 \courantNumber^2 + 128 > 0.
    \end{multline}
    Assume, without loss of generality, that $0 \leq \courantNumber \leq 1$.
    Differentiating the left-hand side in $\mu$, we obtain a second-order equation for the extremal point of this expression.
    The one we are interested in is a minimum explicitly given by 
    \begin{equation*}
        \cosineShorthand_{\textnormal{min}} = \frac{-32 - 124 \courantNumber^2 + 153 \courantNumber^4 - 40 \courantNumber^6 + 16 \courantNumber^8 + 3 \sqrt{1872 \courantNumber^2 - 1640 \courantNumber^4 - 4459 \courantNumber^6 + 8484 \courantNumber^8 - 5392 \courantNumber^{10} + 1216 \courantNumber^{12}}}{16 - 136 \courantNumber^2 + 321 \courantNumber^4 - 136 \courantNumber^6 + 16 \courantNumber^8}.
    \end{equation*}
    By symbolic computations, this point $\cosineShorthand_{\textnormal{min}}$ is in $[-1, 1]$ for $\courantNumber > 0.206...$.
    Below this threshold, the extremal point is on the boundary of $[-1, 1]$: for $\cosineShorthand = -1$, \eqref{eq:StabilityThirdLeapFrog} becomes $-16 (\courantNumber^2 - 1)(2\courantNumber^3+1)^3 > 0$, which is fulfilled. For $\cosineShorthand = 1$, which indeed we do not care about, \eqref{eq:StabilityThirdLeapFrog} reads $432 - 270 \courantNumber^2 > 0$, so $|\courantNumber| < \sqrt{8/5}$, which is true.
    Whenever $\courantNumber > 0.206...$, we plug $\cosineShorthand_{\textnormal{min}}$ into \eqref{eq:StabilityThirdLeapFrog} and solving the associated inequality in $\courantNumber$ using computer algebra provides the condition $|\courantNumber| < 1/2$.
    This achieves the proof.

\end{document}